\documentclass[
11pt,
a4paper,
reqno,
]{amsart}

\usepackage{a4wide}
\usepackage{amsfonts}
\usepackage{amsmath,latexsym,amssymb,amsfonts,amsbsy, amsthm}
\usepackage[
allcolors=blue,
colorlinks=true,
pdfpagelabels,
pagebackref]{hyperref}

\usepackage{mathtools}

\usepackage[usenames]{color}
\usepackage{xspace,colortbl}
\usepackage[textsize=footnotesize, backgroundcolor=yellow, 
linecolor=yellow]{todonotes}

\usepackage[utf8]{inputenc}

\allowdisplaybreaks
\usepackage[T1]{fontenc}

\usepackage{enumitem}

\usepackage{lmodern}
\usepackage{microtype}

\usepackage{bbm}

\usepackage{marginnote}
\theoremstyle{plain}
\newtheorem{theorem}{Theorem}[section]
\newtheorem{definition}[theorem]{Definition}
\newtheorem{corollary}[theorem]{Corollary}
\newtheorem{proposition}[theorem]{Proposition}
\newtheorem{lemma}[theorem]{Lemma}
\newtheorem{remark}[theorem]{Remark}

\numberwithin{theorem}{section} \numberwithin{equation}{section}

\def\R{\mathbb{R}}

\DeclareMathOperator{\arsinh}{arsinh}

\newcommand{\del}{\partial }
\renewcommand{\phi}{\varphi}
\newcommand{\N}{\mathbb{N}}

\newcommand{\cN}{{\mathcal N}}

\newcommand{\embed}{\hookrightarrow}

\newcommand{\B}{{\bf B}}

\newcommand{\eps}{\varepsilon}
\renewcommand{\epsilon}{\varepsilon}

\DeclareOldFontCommand{\bf}{\normalfont\bfseries}{\mathbf}

\usepackage{enumitem}
\makeatletter
\def\namedlabel#1#2{\begingroup
	#2%
	\def\@currentlabel{#2}%
	\phantomsection\label{#1}\endgroup
}
\makeatother
\begin{document}
	\title[
	Second Order Spectral Estimates and Symmetry Breaking
	]{Second Order Spectral Estimates and Symmetry Breaking for Rotating Wave 
	Solutions}
	%
	%
%
%%%%%%%%%%%%%%%%%%%%%%%%%%%%%%%%%%%%%%%%%%%%%%%%%%%%%%%%%%%%%%%%%%%%%%%%%%%%%%%%%%%%%%%%%%%%%%%%%%%%%%%%%%%%
	%
	\author{Joel K\"ubler} 
	\address{Institut f\"ur Mathematik,
	Goethe-Universit\"at Frankfurt,
	Robert-Mayer-Str. 10,
	D-60629 Frankfurt am Main, Germany}
    \email{joel.kuebler.math@gmail.com}
%
%%%%%%%%%%%%%%%%%%%%%%%%%%%%%%%%%%%%%%%%%%%%%%%%%%%%%%%%%%%%%%%%%%%%%%%%%%%%%%%%%%%%%%%%%%%%%%%%%%%%%%%%%%%%
    %
	\date{\today} 

%
%%%%%%%%%%%%%%%%%%%%%%%%%%%%%%%%%%%%%%%%%%%%%%%%%%%%%%%%%%%%%%%%%%%%%%%%%%%%%%%%%%%%%%%%%%%%%%%%%%%%%%%%%%%%
%
	\begin{abstract}
		We consider rotating wave solutions of the nonlinear wave equation
		\begin{equation*}
			\left\{ 
			\begin{aligned}
				\del_{t}^2 v - \Delta v + m v & = |v|^{p-2} v   \quad && 
				\text{in $\R 
					\times \B$} \\
				v & = 0 && \text{on $\R \times \del \B$}
			\end{aligned}
			\right.
		\end{equation*}
		for $2<p<\infty$, $m \in \R$ on the unit disk $\B \subset \R^2$.
		This leads to the study of a reduced equation involving the 
		elliptic-hyperbolic 
		operator $L_\alpha \coloneqq -\Delta + \alpha^2 \del_{\theta}^2$ with
		$\alpha>1$. 
		We find that the structure of the spectrum of $L_\alpha$ strongly 
		depends on the quantity
		\begin{equation*}
			\sigma = 
			\frac{\pi}{\sqrt{\alpha^2- 1} - \arccos \frac{1}{\alpha}} > 0 .
		\end{equation*}
		By giving precise 
		estimates for certain sequences of Bessel function zeros, we can 
		classify the spectrum for all $\alpha>1$ such 
		that $\sigma$ is rational and further find that the existence of 
		accumulation points explicitly depends on arithmetic properties of 
		$\sigma$.
		Using these characterizations, we deduce existence and symmetry 
		breaking results for ground state solutions of the reduced equation,
		extending known results.
	\end{abstract} 
	\maketitle
\section{Introduction}
We study time-periodic solutions of the nonlinear wave equation
\begin{equation} \label{NLKG}
	\left\{ 
	\begin{aligned}
		\del_{t}^2 v - \Delta v + m v & = |v|^{p-2} v   \quad && \text{in $\R 
		\times \B$} \\
		v & = 0 && \text{on $\R \times \del \B$}
	\end{aligned}
	\right.
      \end{equation}
where $2<p<\infty$, $m \in \R$ and we let $\B \subset \R^2$ denote the unit 
disk. 
When $m>0$, this is also known as a nonlinear Klein-Gordon equation, which is typically used to model the dynamics of nonlinear wave propagation in dispersive media and solitary wave phenomena.
The most prominent category of time-periodic solutions is given by standing 
wave solutions, which simplify equation \eqref{NLKG} to either a stationary 
nonlinear Schr\"odinger equation or a nonlinear Helmholtz equation and have been studied
with different methods on $\R^N$ (see e.g., \cite{Evequoz-Weth, Strauss}). 
The corresponding solutions of \eqref{NLKG} are complex-valued with stationary 
amplitudes. However, fewer results on other 
forms of time-periodic solutions are available. 

In particular, there is 
limited knowledge regarding the behavior of nonlinear wave equations in bounded 
domains. 
In one dimension, Rabinowitz~\cite{Rabinowitz} and Brézis, Coron, and 
Nirenberg~\cite{Brezis-Coron-Nirenberg} first studied time-periodic solutions 
satisfying either Dirichlet or periodic boundary conditions via variational 
methods, whereas significantly less is known in higher dimensions. 
The existence of radially 
symmetric time-periodic solutions in balls
was initially studied by Ben-Naoum and Mawhin~\cite{Ben-Naoum-Mawhin} for 
sublinear nonlinearities and has since garnered further interest, e.g. in 
recent works of Chen and Zhang~\cite{Chen-Zhang, Chen-Zhang-2, 
Chen-Zhang-3}.

In the following, we are interested in the existence and properties of 
\emph{rotating wave solutions} of \eqref{NLKG}. These are given as time-periodic solutions originating from 
the ansatz
\begin{equation} \label{eq: Rotating solution ansatz}
	v(t,x)=u(R_{\alpha t} (x)),
\end{equation}
where $R_\theta \in O(2)$ describes a rotation in $\R^2$ with angle $\theta>0$, 
i.e.,
\begin{equation} \label{eq: Rotation definition} 
R_{\theta}(x)= (x_{1} \cos \theta  + x_2 \sin \theta , -x_{1} \sin \theta  + 
x_2 \cos \theta ) \qquad \text{for $x \in \R^2$}.
\end{equation}
The constant $\alpha>0$ in \eqref{eq: Rotating solution ansatz} 
corresponds to the angular velocity of the rotation.

While the study of rotating solutions of \eqref{NLKG} was only recently initiated in \cite{Kuebler-Weth}, solutions of this type are classical in the context of spiral waves, which arise as stable spatio-temporal patterns in reaction-diffusion equations modeling various chemical and physical systems.
Such patterns have been observed experimentally in chemical reactions \cite{Winfree} and catalytic CO oxidation~\cite{Nettesheim-et-al}, motivating subsequent mathematical analyses; see for example the works of Sandstede and Scheel based on
dynamical systems techniques (we refer to~\cite{Sandstede-Scheel-Book} for an overview).
Another source of interest in rotating solutions comes from mathematical observations.
On the one hand, when the domain is a more general manifold such solutions 
can be interpreted as generalized traveling waves, see \cite{Taylor}.
On the other hand, for stationary problems, a related ansatz can also be used to find spiral-shaped solutions, see \cite{Agudelo-Kuebler-Weth} for results for nonlinear Schr\"odinger equations in particular.

Letting $\theta$ denote the angular variable in 
two-dimensional polar coordinates, the ansatz \eqref{eq: Rotating solution ansatz} 
reduces \eqref{NLKG} to
\begin{equation} \label{Reduced equation}
	\left\{ 
	\begin{aligned} 
		-\Delta u + \alpha^2 \del_{\theta}^2 u +m u & = |u|^{p-2} u \quad 
		&& \text{in $\B$} \\
		u & = 0 && \text{on $\del \B$}
	\end{aligned}
	\right.
\end{equation}
where $\del_\theta = x_{1} \del_{x_2} - x_{2} \del_{x_{1}}$ then corresponds to 
the angular derivative.
Radial solutions of \eqref{Reduced equation} are well understood, since they 
can be related to the case $\alpha=0$, i.e., the classical Lane-Emden equation. 
These lead to stationary solutions of \eqref{NLKG}, however, and we are thus 
interested in the existence of nonradial solutions 
of \eqref{Reduced equation}. Furthermore, it is natural to ask whether ground 
states, i.e., solutions that minimize a suitable energy functional, can be 
defined for this problem and to classify whether ground states are radial or 
nonradial.

In the cases $\alpha < 1$ and $\alpha=1$, the operator
\[
L_\alpha \coloneqq -\Delta + \alpha^2 \del_{\theta}^2 
\]
is elliptic
and degenerate-elliptic, respectively. 
In these cases, ground states can be defined using the classical energy 
functional on the Sobolev space $H_0^1(\B)$ or a suitable modification thereof 
for $\alpha<1$ and $\alpha=1$, respectively. In these settings, the existence and symmetry 
breaking for ground states has been studied extensively in \cite{Kuebler-Weth}. 
The results therein essentially rely on the observation that the associated Rayleigh
quotient can be related to a degenerate Sobolev inequality on the half-space
and its critical exponent.

In the case $\alpha>1$, however, $L_\alpha$ is of mixed type, since it is 
elliptic in $B_{1/\alpha}(0)$, parabolic on $\del B_{1/\alpha}(0)$ and 
hyperbolic in $\B \setminus \overline{B_{1/\alpha}(0)}$. As a consequence, 
methods from the elliptic setting cannot easily be extended to this case. 
Instead, first results for certain values of $\alpha>1$ were obtained in 
\cite{Kuebler} 
by studying the spectrum of $L_\alpha$. Using eigenfunctions of the laplacian 
on $\B$, it can be shown that there exists an orthonormal basis of 
eigenfunctions of $L_\alpha$, and the Dirichlet eigenvalues of $L_\alpha$ are 
explicitly given by
\begin{equation}\label{eq:spectrum_definition}
\Sigma_\alpha := \left\{j_{\ell,k}^2-\alpha^2 \ell^2: \ell \in \N_0, k \in \N 
\right\},
\end{equation}
where $\ell \in \N_0$, $k \in \N$ and $j_{\ell,k}$ denotes the $k$-th zero of 
the Bessel function of the first kind $J_{\ell}$.
The main results of \cite{Kuebler} established that there exists a sequence 
\((\alpha_n)_n\) tending to infinity, such that $\Sigma_\alpha$ consists of 
isolated eigenvalues for $\alpha=\alpha_n$. Moreover, the estimate
\begin{equation}\label{eq:eigenvalue_lowerbound}
|j_{\ell,k}^2 - \alpha_n^2 \ell^2 |\geq c_n  j_{\ell,k}
\end{equation}
for nonzero eigenvalues was established.

The goal of the present paper is to significantly improve these results,
based on new estimates for \(j_{\ell,k}\).
We first note the main difficulties in obtaining estimates similar to
\eqref{eq:eigenvalue_lowerbound}.
Writing
\[
j_{\ell,k}^2 - \alpha^2 \ell^2  
= (j_{\ell,k} + \alpha \ell) \ell \left(\frac{j_{\ell,k}}{\ell}- \alpha\right),
\]
it becomes clear that sequences of eigenvalues can only grow slower than 
quadratically in \(\ell\), or even tend to zero, if \({j_{\ell,k}}/{\ell}- 
\alpha \to 0\) along some sequence of \(\ell,k\). Hence, analyzing such 
behavior plays a crucial role.
From classical estimates, it can be deduced that this can only happen for 
sequences \((\ell_i)_i,(k_i)_i\) such that \({\ell_i}/{k_i}\) converges as 
\(i \to \infty\). In fact, as a consequence of results by Elbert and 
Laforgia~\cite{Elbert-Laforgia}, it can be shown that
\begin{equation}\label{eq:sigma-def-intro}
\sigma = \sigma(\alpha) =
\frac{\pi}{\sqrt{\alpha^2- 1} - \arccos \frac{1}{\alpha}} > 0 
\end{equation}
is the unique value such that
\begin{equation}\label{eq:intro_critical_convergence}
 \frac{j_{\sigma k,k}}{\sigma k}- \alpha \to 0 
\end{equation}
as \(k \to \infty\), and we further find that \(\sigma\) is the only value such 
that sequences satisfying \(\lim_{i \to \infty}{\ell_i}/{k_i}=\sigma\) can 
yield an accumulation point in the spectrum.

Importantly, the right hand side of \eqref{eq:sigma-def-intro} has the following properties.
\begin{lemma}\label{lemma:sigma_map_is_bijective}
  The function
  \[
    (1,\infty) \to (0,\infty), \qquad \alpha \mapsto \frac{\pi}{\sqrt{\alpha^2- 1} - \arccos \frac{1}{\alpha}}
  \]
  is smooth, strictly decreasing and bijective.
\end{lemma}
See Section~\ref{Section: Preliminaries} below for more details.
In particular, it follows that for any \(\sigma_0>0\) there exists \(\alpha_0>1\)
such that \(\sigma(\alpha_0) = \sigma_0\).
Crucially, the inverse function $\sigma^{-1}$ plays a central role in 
estimating the rate of 
convergence in \eqref{eq:intro_critical_convergence}. This was first observed
in \cite{Kuebler}, where convergence properties of \( {j_{\sigma 
    k,k}}/{\sigma k}- \alpha\)
were used to obtain estimates of the form \eqref{eq:eigenvalue_lowerbound}.
More specifically, one of the main results established the inequalities 
\begin{equation}\label{eq:intro_iota_bounds}
  - \exp \left( \left(\frac{1}{3} + \eps \right) x  \right) \frac{\pi}{4 k}
  \leq  \frac{j_{x k,k}}{k} - x \sigma^{-1}(x)
  \leq - (1-\eps) \frac{\pi}{4 k}
\end{equation}
for sufficiently large $k$. These estimates are then used to study the behavior 
of sequences ${j_{\ell_i,k_i}}/{\ell_i}- \alpha$ satisfying \(\lim_{i \to 
\infty}{\ell_i}/{k_i}=\sigma\) and determine the existence of accumulation 
points. In fact, the main result states that the spectrum of $L_\alpha$ has no 
accumulation points and that an estimate of the form 
\eqref{eq:eigenvalue_lowerbound} is valid, provided $\alpha$ is chosen such that
\(\sigma(\alpha) = {1}/{n}\) for sufficiently large $n \in \N$. 
This yields a sequence $\alpha_n \to \infty$ for which existence results for 
\eqref{Reduced equation} are then provided.
The argument 
used therein fundamentally requires that $\sigma(\alpha)$ is a rational number 
and, as pointed out in \cite[Remark 4.3(ii)]{Kuebler}, the proof can only 
be extended to $\alpha$ such that \(\sigma(\alpha) = {m}/{n}\) with $m=1,2,3$.

The main goal of the present paper is to extend these results and further 
classify the structure of the spectrum for all rational $\alpha>1$, and give 
according existence results for \eqref{Reduced equation}.
This is based on a better understanding of the convergence in 
\eqref{eq:intro_critical_convergence} obtained by improving 
\eqref{eq:intro_iota_bounds} to a precise asymptotic expansion up to second 
order. More specifically, our first main result reads as follows.
\begin{theorem}\label{theorem:intro_second_order_expansion}
	Let 
	$x>0$.
	Then there exists $\zeta_x \in \R $ such that
	\[
	\frac{j_{xk,k}}{k} - x \sigma^{-1}(x) = - \frac{\pi}{4k} 	
	\frac{\sigma^{-1}(x)}{\sqrt{(\sigma^{-1}(x))^2-1}} + \frac{\zeta_x}{k^2} + 
	o\left(\frac{1}{k^2}\right)
	\]
	as $k \to \infty$. 
      \end{theorem}
      The constant \(\zeta_x\) is given in terms of \(\sigma^{-1}(x)\) and 
      integrals over
      modified Bessel functions, see Theorem~\ref{theorem:second_order_expansion} below for details.
      We further show that \(\zeta_x\) is continuous
      and satisfies
      \[
      -\infty = \lim_{x \to \infty} \zeta_x < 0 < \lim_{x \to 0} \zeta_x = 
      \frac{1}{8 \pi} .
      \]
      Additionally, there exists a unique $x_0>0$ such that $\zeta_{x_0} = 0$ 
      and $\zeta_x$ is nonzero otherwise.
      
      The proof of Theorem~\ref{theorem:intro_second_order_expansion} uses the fact that 
      $\frac{j_{xk,k}}{k}$ and $x \sigma^{-1}(x)$ can be characterized via integro-differential
      equations involving modified Bessel functions of the second kind. 
      Through a careful technical analysis we can study the convergence of the 
      respective problems via several estimates and establish a suitable Taylor series.
      
      Importantly, Theorem~\ref{theorem:intro_second_order_expansion} allows us 
      to study
      the order of convergence in \eqref{eq:intro_critical_convergence} in much 
      more detail.
      We use this in our next main result, yielding an estimate of the form
      \eqref{eq:eigenvalue_lowerbound} whenever
      \(\sigma(\alpha)\) is rational and its numerator and denominator satisfy 
      certain divisibility conditions.
\begin{theorem}\label{theorem:intro:main_result}
	 Let $\alpha>1$ such that $\sigma = \sigma(\alpha)$ is rational and let 
	 $q_1,q_2 \in \N$ be coprime such that $\sigma = \frac{q_1}{q_2}$. Additionally, 
	 assume that at least one of the following two conditions holds:
	 \begin{enumerate}
 		\item[\namedlabel{condition:C1}{\textit{(C1)}}]
	 	4 is not a divisor of $q_1$.
	 	\item[\namedlabel{condition:C2}{\textit{(C2)}}]
	 	$q_2$ is even.
	\end{enumerate}
Then, the following statements hold for $\Sigma$ as given by \eqref{eq:spectrum_definition}:
	 \begin{itemize}
	 	\item[(i)] $\Sigma$ consists of eigenvalues with 
	 	finite multiplicity.
	 	\item[(ii)] There exists $c>0$ 
	 	such that for each 
	 	$\ell \in \N_0$, $k \in \N$ we either have $j_{\ell,k}^2-\alpha^2 
	 	\ell^2=0$ or
		\[
	 	|j_{\ell,k}^2-\alpha^2 \ell^2| \geq c j_{\ell,k}.
	 	\] 
	 	\item[(iii)] $\Sigma$ has no finite accumulation points.
	 \end{itemize}
\end{theorem}
In particular, condition \ref{condition:C1} is satisfied for the sequence 
\(\alpha_n\) 
considered
in \cite{Kuebler}, hence this result is strictly stronger.  

In the proof we essentially use 
Theorem~\ref{theorem:intro_second_order_expansion} to analyze $j_{\ell_i,k_i}$ 
for sequences $(\ell_i)_i, (k_i)_i$ such that \(\lim_{i \to 
\infty} {\ell_i}/{k_i}=\sigma\) with $\sigma$ given by 
\eqref{eq:sigma-def-intro} in much more detail. 
We reduce the problem to sequences such that 
$\ell_i = \sigma k_i - \delta_i$ for numbers $\delta_i>0$ and further estimates 
then show that $j_{\ell_i,k_i} - \alpha \ell_i$ can only converge if 
$\delta_i \to {\sigma}/{4}$, hence the assumption that $\sigma$ is rational then 
necessitates $\delta_i = {\sigma}/{4}$. This, however, can be brought to a 
contradiction with \ref{condition:C1} or \ref{condition:C2} using elementary 
divisibility properties.

Surprisingly, it turns out that the divisibility conditions in 
Theorem~\ref{theorem:intro:main_result}
are in fact necessary
in the sense that the spectral properties of \(L_\alpha\) are different if the conditions
\ref{condition:C1}, \ref{condition:C2} are not satisfied. More specifically, we 
have the following.

\begin{theorem}  \label{theorem:intro:main_result3}
	Let $\alpha>1$ such that $\sigma = \sigma(\alpha)>0$ is rational and let 
	$q_1, q_2 \in \N$ be coprime such that $\sigma = \frac{q_1}{q_2}$. Additionally, 
	assume that 
	\begin{itemize}
		\item[\namedlabel{condition:C3}{\textit{(C3)}}]
		4 is a divisor of $q_1$ and $q_2$ is odd.
              \end{itemize}
              Then
              \[
              \Sigma_* := \left\{j_{\ell,k}^2-\alpha^2 \ell^2:
                k \in \N, 
                \ell = \sigma k - {\sigma}/{4} \in \N  \right\}
              \]
              is nonempty and \(\Sigma_* \subset \Sigma\).
              In particular, $\Sigma_*$ consists of a single sequence and this 
              sequence converges to \(2 \alpha \sigma \zeta_\sigma\).
              
              Moreover, the remaining eigenvalues have the following properties:
	\begin{itemize}
		\item[(i)] $\Sigma \setminus \Sigma_*$ is unbounded from above and below and 
		exclusively consists of eigenvalues with 
		finite multiplicity.
		\item[(ii)] There exists $c>0$ 
		such that for each 
		$\ell \in \N_0$, $k \in \N$ such that $j_{\ell,k}^2-\alpha^2 \ell^2 \in \Sigma \setminus \Sigma_*$, we either have $j_{\ell,k}^2-\alpha^2 
		\ell^2=0$ or
		\begin{equation}
			|j_{\ell,k}^2-\alpha^2 \ell^2| \geq c j_{\ell,k}.
		\end{equation} 
		\item[(iii)] $\Sigma \setminus \Sigma_*$ has no finite 
		accumulation points.
	\end{itemize}
      \end{theorem}

      \begin{remark}
      Using Theorems~\ref{theorem:intro:main_result} and \ref{theorem:intro:main_result3},
      we can thus characterize the spectrum of \(L_\alpha\) as given
      in \eqref{eq:spectrum_definition} whenever \(\alpha > 1\) is chosen
      such that \(\sigma(\alpha)\) is rational.
      The set of such \(\alpha > 1\) is a dense subset of \((1,\infty)\).
      Indeed, for any \(\sigma_0 \in \mathbb{Q}\), Lemma~\ref{lemma:sigma_map_is_bijective}
      yields the existence of \(\alpha_0 >1\) such that \(\sigma(\alpha_0) = \sigma_0\)
      and that the map \(\sigma \mapsto \sigma^{-1}(\alpha)\) is continuous.
      \end{remark}
The fact that we find more detailed information on the spectrum if $\sigma$ is 
rational is reminiscent of the study of radially symmetric
time-periodic solutions of \eqref{NLKG} on general balls. Indeed, the spectrum
of the radial wave operator is also nicely behaved whenever the ratio 
between the radius of the ball and the period length 
is rational, see e.g. \cite{Mawhin-Survey} and the references therein. 
This is generally related to small divisor problems, which have also been observed
to occur in other elliptic-hyperbolic problems, see e.g. \cite{Dong-Li}.
To our knowledge, however, an explicit dependence of the fundamental structure 
of the spectrum on arithmetic properties has not been observed in either setting. 

Finally, Theorem~\ref{theorem:intro:main_result} allows us to extend the 
existence result for ground state solutions of
\eqref{Reduced equation} and their nonradiality stated in \cite{Kuebler}.
More precisely, we have the following.
\begin{theorem} \label{theorem:intro:main_result2}
	Let $p \in (2,4)$ and let \(\alpha>1\) such that $\sigma(\alpha)$ is 
	rational
	and satisfies one of the conditions \ref{condition:C1}, \ref{condition:C2}. 
	Then the following properties hold:
	\begin{itemize}
		\item[(i)]
		For any $m \in \R$ there exists a ground state solution 
		of \eqref{Reduced equation}.
		\item[(ii)]
		There exists $m_0>0$ 
		such that the ground state solutions of \eqref{Reduced equation} are 
		nonradial for $m >m_0$. 
	\end{itemize}
\end{theorem}
In the case \ref{condition:C3}, on the other hand, the existence of ground 
states cannot directly be adapted from 
the methods introduced in \cite{Kuebler} due to the presence of an accumulation 
point in the spectrum. Moreover, the structure of spectrum in the case of 
irrational $\sigma$ remains open as well and we leave these questions for 
future 
research.

The paper is organized as follows. 
In Section~\ref{Section: Preliminaries} we first collect known properties of 
zeros of Bessel functions, as well as several useful estimates for modified 
Bessel functions of the second kind. These technical tools play a crucial role 
in Section~\ref{Section: New estimates}, where we successively improve 
\eqref{eq:intro_iota_bounds} to first and second order with exact constant 
which results in the proof of 
Theorem~\ref{theorem:intro_second_order_expansion}.
We then further discuss the properties of $\xi_x$ and related results in 
Section~\ref{sec:auxiliary-results}.
In Section~\ref{Section:Spectral-Characterization}, we turn to the results for 
rotating wave solutions and characterize the spectrum of $L_\alpha$ and prove 
Theorems~\ref{theorem:intro:main_result} and \ref{theorem:intro:main_result3}. 
Additionally, we elaborate on the construction of ground states and prove 
Theorem~\ref{theorem:intro:main_result2}.
%
%
%
%%%%%%%%%%%%%%%%%%%%%%%%%%%%%%%%%%%%%%%%%%%%%%%%%%%%%%%%%%%%%%%%%%%%%%%%%%%%%%%%%%%%%%%%%%%%%%%%%%%%%%%%%%%%
%         ACKNOWLEDGMENTs
%
\subsection*{Acknowledgments}
The author thanks the anonymous referee for their helpful comments and suggestion which
helped to improve the paper.
Moreover, the author also thanks Fu Guoqing for helpful comments which helped to improve the proof of Lemma~\ref{lemma:new_iota_bounds}.
%
%
%
%%%%%%%%%%%%%%%%%%%%%%%%%%%%%%%%%%%%%%%%%%%%%%%%%%%%%%%%%%%%%%%%%%%%%%%%%%%%%%%%%%%%%%%%%%%%%%%%%%%%%%%%%%%%
%
%
%
%
\section{Preliminaries}\label{Section: Preliminaries}
From the classical study of the eigenvalue problem for the Laplacian on $\B$,
it is well-known that in polar coordinates and with suitable constants 
$A_{\ell,k}, 
B_{\ell,k}>0$, the functions 
\[
\begin{aligned}
	\phi_{\ell,k}(r,\theta ) & \coloneqq A_{\ell,k} \cos(\ell \theta) 
	J_{\ell}(j_{\ell,k}r) \\	
	\psi_{\ell,k}(r,\theta ) & \coloneqq B_{\ell,k} \sin(\ell \theta) 
	J_{\ell}(j_{\ell,k}r) 
\end{aligned} 
\]
constitute an orthonormal basis of $L^2(\B)$ for $\ell \in \N_0, k \in \N$.
Importantly, $\phi_{\ell,k}, \psi_{\ell,k}$ are eigenfunctions of $L_{\alpha}$
corresponding to the eigenvalue $j_{\ell,k}^2 - \alpha^2 \ell^2$ and hence 
the Dirichlet eigenvalues of $L_\alpha$ are given by 
\[
\Sigma_\alpha = \left\{ j_{\ell,k}^2-\alpha^2 \ell^2 : \ \ell \in \N_0, \, k \in \N \right\}.
\]

\subsection{Zeros of Bessel functions}
In the following we collect some useful properties of zeros of Bessel 
functions. 
Firstly, we recall the following classical expansion when the order $\nu$ remains fixed:
\begin{lemma}\label{lemma:mcmahon}
	\textbf{(McMahon's expansion)}
	Let $\nu \in \R$ be fixed. Then
	\[
	j_{\nu,k} = k \pi + \frac{\pi}{2} \left(\nu - \frac{1}{2}\right) - 
	\frac{4\nu^2-1}{8\left(k\pi + \frac{\pi}{2} \left(\nu - 
	\frac{1}{2}\right)\right)} + O \left(\frac{1}{k^3}\right) \qquad \text{as 
	$k \to \infty$.}
	\]
\end{lemma}

\begin{remark}\label{remark:on-mcmahon}
	Note that the coefficients in McMahon's expansion diverge as $\nu \to \infty$.
	Consequently, Lemma ~\ref{lemma:mcmahon} can only be used to estimate $j_{xk,k}/k$ 
	as $k \to \infty$ when $x=0$.
\end{remark}

Next, we recall the following result by Elbert and Laforgia which characterizes
the asymptotic behavior of $j_{\nu,k}$ as $\nu, k \to \infty$ while the ratio $\nu/k$ remains fixed.
\begin{theorem} \label{Theorem: Elbert-Laforgia Asymptotic Relation}
	{\bf(\cite{Elbert-Laforgia})} \\
	Let $x>-1$ be fixed. Then 
	\[
	\lim_{k \to \infty} \frac{j_{x k,k}}{k} \eqqcolon \iota(x) 
	\]
	exists. Moreover, $\iota(x)$ is given by
	\[
	\iota(x) = \begin{cases}
		\pi , \qquad & \text{$x=0$,} \\
		\frac{x}{\sin \phi} & \text{$x \neq 0$}
	\end{cases}
	\]
	where $\phi=\phi(x) \in [-\frac{\pi}{2},\frac{\pi}{2}]$ denotes the unique solution of
	\begin{equation} \label{eq: iota definition identity} 
		\frac{\sin \phi}{\cos \phi - (\frac{\pi}{2}-\phi) \sin \phi} = \frac{x}{\pi} .
	\end{equation}
\end{theorem}
In \cite{Kuebler}, we noted the following useful property of a function related
to $\iota$.
\begin{lemma} \label{Lemma: Monotonicity and Inverse for iota quotient}
	\textbf{(\cite{Kuebler})}
	The map 
	\begin{equation*}
		f: (0,\infty) \to \R, \qquad f(x)= \frac{\iota(x)}{x}
	\end{equation*}
	is strictly decreasing and satisfies
	\[
	\lim_{x \to 0}  f(x) = \infty, \qquad \lim_{x \to \infty} f(x) = 1 .
	\]
	Moreover, its inverse is explicitly given by
	\[
	f^{-1}: (1,\infty) \to \R, \qquad f^{-1}(y) 
	= \frac{\pi}{\sqrt{y^2- 1} - \left(\frac{\pi}{2} - \arcsin \frac{1}{y}\right) } .
	\]
\end{lemma}
Due to the identity $\arccos x + \arcsin x = \frac{\pi}{2}$, 
$f^{-1}$ can also be characterized as
\begin{equation}\label{eq:f_inv-characterization}
f^{-1}(y) = \frac{\pi}{\sqrt{y^2- 1} - \arccos \frac{1}{y}}, \qquad y>1. 
\end{equation}
Note that the functions $\sigma$ and 
$\sigma^{-1}$ used in the statement of our main results earlier correspond to 
the functions $f^{-1}$ and $f$, respectively. Since most of the technical 
results will be stated in terms of $\sigma^{-1}$, we will use the notation in 
terms of $f$ in the following, which is also more consistent with the 
results given in \cite{Kuebler}.
\begin{remark}
Clearly, $f^{-1}$ is a smooth function and satisfies
\begin{equation} \label{eq:f_inv_derivative}
	(f^{-1})'(y)= - \frac{\pi \sqrt{y^2-1}}{y\left(\sqrt{y^2- 1} - \arccos 
	\frac{1}{y}\right)^2}
	= - \frac{\sqrt{y^2-1}}{\pi y} (f^{-1}(y))^2
\end{equation}
and $(f^{-1})'$ is thus negative, and strictly increasing with $\lim_{y \to 
\infty}(f^{-1})'= 0$.
Similarly,
\[
\begin{aligned} 
(f^{-1})''(y)&= \frac{2\pi (y^2-1)}{y^2 \left(\sqrt{y^2- 1} - \arccos 
	\frac{1}{y}\right)^3}
-
\frac{\pi}{\left(\sqrt{y^2- 1} - \arccos 
	\frac{1}{y}\right)^2} \frac{1}{y^2\sqrt{y^2-1}} \\
&= \frac{2}{\pi^2} (f^{-1}(y))^3 - \frac{2}{\pi^2 y^2} (f^{-1}(y))^3 - 
\frac{1}{\pi} (f^{-1}(y))^2 \frac{1}{y^2\sqrt{y^2-1}} .
\end{aligned} 
\]
\end{remark}
Next, we recall that the convergence result stated in Theorem~\ref{Theorem: 
Elbert-Laforgia Asymptotic Relation} was based on the Watson  integral  formula 
\cite[p. 508]{Watson}, which states that for fixed $k \in \N$ the function $\nu 
\mapsto 
j_{\nu,k}$ satisfies 
\[
\frac{d}{d\nu} j_{\nu,k} = 2 j_{\nu,k} \int_0^\infty K_0 (2  j_{\nu,k} 
\sinh(t)) e^{-2 \nu t} \, dt ,
\]
where $K_0$ denotes the modified Bessel function of the second kind of order zero. 
It then follows that the function 
\[
\iota_k(x) \coloneqq \frac{j_{kx,k}}{k}
\]
satisfies
\begin{equation} \label{eq: F_k definition}
	\frac{d}{dx} \iota_k(x)= 2 \iota_k \int_0^\infty K_0 \left( t 2 \iota_k  
	\frac{\sinh\left( \frac{t}{k} \right)}{\left(\frac{t}{k}\right)} \right) 
	e^{-2x t} \, dt \eqqcolon F_k(\iota_k,x) 
\end{equation}
for $k \in \N$ and $x \in (-1,\infty)$.
In \cite{Elbert-Laforgia} it is then shown that $\iota_k$ converges pointwisely
to the solution of 
\begin{equation} \label{eq: G definition}
	\left\{ 
	\begin{aligned} 
	\frac{d}{dx} \iota(x) &= 2 \iota \int_0^\infty K_0 \left( t 2 \iota  
	\right)  e^{-2x t} \, dt  \eqqcolon G(\iota,x) \\ 
	\iota(0) &=\pi ,
	\end{aligned}
    \right.
\end{equation}
which is precisely given by the function $\iota$ discussed in 
Theorem~\ref{Theorem: Elbert-Laforgia Asymptotic Relation}. Moreover, it is 
shown that
\begin{equation} \label{eq: iota - iota_k inequality}
	\iota_k(x) < \iota(x)
\end{equation}
holds for all $k \in \N$.

In \cite{Kuebler}, we gave the following bounds for this convergence for
$x>0$.
\begin{lemma} \label{lemma:old_iota_bounds}
	\textbf{(\cite{Kuebler})}
	For any $x>0$ and $\eps>0$ there exists $k_0 \in \N$ such that
	\[
	- \exp \left( \left(\frac{1}{3} + \eps \right) x  \right) \frac{\pi}{4 k} 
	\leq  \frac{j_{x k,k}}{k} - \iota(x)  \leq - (1-\eps) \frac{\pi}{4 k} 
	\]
	holds for $k \geq k_0$.
\end{lemma}

The proof of this result is based on suitable estimates for the convergence of the integral
on the right hand side of \eqref{eq: F_k definition}. 
We note that the proof given in \cite{Kuebler} can be improved to give the following stronger result.

\begin{lemma} \label{lemma:old_iota_bounds-uniform}
	For any $\eps>0$ there exists $k_0 \in \N$ such that
	\[
	- \exp \left( \left(\frac{1}{3} + \eps \right) x  \right) \frac{\pi}{4 k} 
	\leq  \frac{j_{x k,k}}{k} - \iota(x)  \leq - (1-\eps) \frac{\pi}{4 k} 
	\]
	holds for $k \geq k_0$ and $x>0$.
\end{lemma}
\begin{proof}[Proof Sketch]
	Essentially the proof of \cite[Lemma 3.3]{Kuebler} can be repeated with the following slightly different estimates:
	
	In order to make the upper bound uniform, we essentially need to have $\left|\frac{x}{\iota_k(x)}\right| < 1$
	for all $x>0$ and $k \in \N$. To this end, note that \cite[Proposition 2.1(i)]{Kuebler} implies
	\[
	j_{xk,k} > xk + \frac{|a_k|}{2^{1/3}} (xk)^{1/3}
	\]
	where $a_k$ denotes the $k$-th negative zero of the Airy function.
	In particular, this implies
	\[
	\frac{j_{xk,k}}{k} > x
	\]
	as claimed.
	
	In order to make the lower bound uniform, a different estimate needs to be used in place of \cite[Equation (3.11)]{Kuebler}.
	More specifically, letting $t_k = k^\frac{1}{3}$, we have the alternative estimate
	\[
	\int_{t_k}^\infty K_0(t) e^{-\frac{xt}{y}} \, dt \leq \int_{t_k}^\infty K_0(t) \, dt
	\leq \sqrt{\frac{\pi}{2t_k}} \int_{t_k}^\infty e^{-t} \, dt 
	= \sqrt{\frac{\pi}{2t_k}} e^{-t_k} .
	\]
	For given $\delta \in (0,1)$ there hence exists $k_0 \in \N$ such that
	\[
	\int_{t_k}^\infty K_0(t) e^{-\frac{xt}{y}} \, dt \leq \frac{\delta}{k^\frac{4}{3}}
	\]
	holds for all $y>x>0$ and $k \geq k_0$. 
	The rest of the proof proceeds normally, making note of the fact that $k_0$ only depended on $x$ in these 
	two estimates.
\end{proof}

\subsection{Modified Bessel function of the second kind}
In order to improve the convergence result stated in Lemma~\ref{lemma:old_iota_bounds-uniform},
we further study the involved modified Bessel functions and related integrals.

\begin{lemma} \label{lemma:modified_bessel}
	For $x>0$, the modified Bessel function of the second kind of order $\nu 
	\geq 0$ is given by 
	\[
	K_\nu(x) = \int_0^\infty \exp(-x \cosh t) \cosh(\nu t) \, dt.
	\]
	The following properties hold:
	\begin{itemize}
		\item[(i)]
		For any $\gamma>0$, there exists a constant $C_{\nu,\gamma}>0$ such that
		\[
		0 < K_\nu(x) \leq C_{\nu,\gamma} \left(1+ \frac{1}{x^{\nu+\gamma}}  
		\right)e^{-x} \qquad \text{for $x>0$.}
		\]
		\item[(ii)]
		The following relations are valid:
		\[
		\begin{aligned}
			K_0' & = -K_1 \\
			K_n' & = -\frac{1}{2}(K_{n-1} + K_{n+1}) \qquad \text{for $n \in 
			\N$.}
		\end{aligned}
		\] 
		\item[(iii)]
		For $\mu > \nu$ and $a \in (-1,1)$, it holds that
\[
\int_0^\infty \exp(-at) K_\nu(t) t^{\mu-1} \, dt = 
\sqrt{\frac{\pi}{2}} \Gamma(\mu-\nu) \Gamma(\mu+\nu) 
(1-a^2)^{-\frac{\mu}{2} + \frac{1}{4} }
P^{\frac{1}{2}-\mu}_{\nu 
-\frac{1}{2}}(a),
\]
where $P_\lambda^\beta$ denotes the Legendre function of the first kind.
	\end{itemize}
\end{lemma}
We note that the Legendre function of the first kind is given by
\begin{equation}\label{eq:legendre_function_def}
P^{\frac{1}{2}-\mu}_{\nu-\frac{1}{2}}(a)
= \frac{1}{\Gamma\left(\frac{1}{2}+\mu\right)} \left(\frac{1+a}{1-a} \right)^{\frac{1}{4}-\frac{\mu}{2}}
\mathbf{F}\left(\frac{1}{2}- \nu, \nu + \frac{1}{2}; \frac{1}{2}+\mu; \frac{1-a}{2}\right),
\end{equation}
where \(\mathbf{F}\) denotes the hypergeometric function. Moreover,
\begin{equation}\label{eq:hypergeometric_function_property}
\mathbf{F}\left(\frac{1}{2}- \nu, \nu + \frac{1}{2}; \frac{1}{2}+\mu; 0\right)=1 .
\end{equation}
\begin{proof}
	(i): We first write
	\[
	\begin{aligned}
	 \exp(-x \cosh t) \cosh(\nu t) & = \frac{1}{2} \exp(-x \cosh t) (\exp(\nu 
	 t) 	 + \exp(-\nu t)) \\
	 & =  \frac{1}{2} (\exp(-x \cosh t +\nu t ) + \exp(-x \cosh t -\nu t ) \\
	 & \leq \exp(-x \cosh t +\nu t ) .
	 \end{aligned}
	\]
	For fixed $\gamma>0$, it is straightforward to see that 
	\[
	\begin{aligned} 
	-x \cosh t +\nu t 
	& \leq -\gamma t + (\nu+\gamma) \arsinh\left(\frac{\nu+\gamma}{x}\right) - 
	x 
	\sqrt{1+\left(\frac{\nu+\gamma}{x}\right)^2} \\
	& \leq -\gamma t +  (\nu+\gamma) \arsinh\left(\frac{\nu+\gamma}{x}\right) - 
	x 
	\end{aligned} 
	\]
	and hence
	\[
	\begin{aligned} 
	K_\nu(x) &\leq \int_0^\infty \exp\left( -\gamma t +  (\nu+\gamma) 
	\arsinh\left(\frac{\nu+\gamma}{x}\right) - x  \right) \, dt \\
	& = 
	\frac{1}{\gamma} \exp\left(  
	(\nu+\gamma) 
	\arsinh\left(\frac{\nu+\gamma}{x}\right) - x  \right).
	\end{aligned}
	\]
	For $x>\nu+\gamma$, it follows that
	\[
	K_\nu(x) \leq \frac{\exp((\nu+\gamma)\arsinh(1))}{\gamma} e^{-x},
	\]
	while for $0<x \leq \nu+\gamma$, we find that
	\[
	\begin{aligned} 
	K_\nu(x) & \leq 
                  \gamma^{-1 } {\left(\sqrt{1+\left(\frac{\nu+\gamma}{x}\right)^2}
                   +\frac{\nu+\gamma}{x}\right)^{\nu+\gamma}}
	 e^{-x}
	 \leq \frac{(2(\nu+\gamma)^{\nu+\gamma})}{\gamma} 
	\frac{e^{-x}}{x^{\nu+\gamma}} .
	\end{aligned} 
	\]
	Combining these two estimates yields the claim.
	
	The formulas stated in (ii), (iii) can be found in \cite[p. 79]{Watson} and 
	\cite[p. 388]{Watson}, respectively.
\end{proof}
Due to the relation \eqref{eq: F_k definition} it will be helpful to have 
formulas for integrals involving $K_0$ and exponential terms, as stated in
the following results.

\begin{lemma} \label{lemma:g_properties}
	For $\alpha>1$, it holds that
	\[
	\int_0^{\infty} K_0(t)  {\exp\left(-\frac{t}{\alpha}\right)} \, dt = 
	g(\alpha)
	\]
	where $g:(1,\infty) \mapsto \R$ is given by
    \[
        g(t) =  \frac{\arccos \frac{1}{t}}{\sqrt{1-\frac{1}{t^2}}} .
	\]
	Moreover, $g$ is smooth and we have
	\[
		g'(t) 
		 = \frac{\sqrt{t^2-1} - \arccos 
			\frac{1}{t}}{\left(t^2-1\right)^\frac{3}{2}} 
	\]
	and
	\begin{equation}\label{eq:g-second-derivative-def}
	g''(t)
	= \frac{3t\arccos\frac{1}{t} - \sqrt{1-\frac{1}{t^2}} - 			
		2t\sqrt{t^2-1} }{(t^2-1)^\frac{5}{2}} .
	\end{equation}
	In particular, $g'$ is positive and strictly decreasing while $g''$ is negative and 
	strictly increasing.
	Additionally, $g',g''$ tend to zero as $t \to \infty$, and $g,g',g''$ can 
	be continuously extended in $t=1$ by setting
	\[
	g(1)  = 1, \qquad g'(1) = \frac{1}{3} , \quad 	g''(1) = - \frac{2}{5} ,
	\]
	and $|g'|,|g''|$ attain their maxima at $t=1$. 
\end{lemma}
\begin{proof}
  The integral identity is given by \cite[p. 388]{Watson}. The remaining assertions follow
  from straightforward computations.
\end{proof}
We will also encounter the composition of $g''$ with $f$, for which we provide the following estimates.
\begin{lemma}
	\label{lemma:g-f-integrability}
	Let \(a>0\). Then the the function
	\[
	(0,\infty) \to \R, \qquad x  \mapsto \frac{g''(a f(x))}{x^2} .
	\]
	is integrable on $(0,b)$ for any $b>0$.
	Moreover, there exists an increasing function $\kappa:(0,\infty) \to \R$ such that 
	\[
	\left|\frac{g''(a f(x))}{x^2} \right| \leq \kappa(x) \qquad \text{for $x>0$.}
	\]
\end{lemma}

\begin{proof}
	It suffices to prove the existence of $\kappa$, as this also implies the claimed integrability properties. 
	
	Without loss of generality, we may assume $a=1$. 
	By Lemma~\ref{lemma:g_properties}, the function $g''$ is negative, strictly increasing on $(1, \infty)$ 
	with $\lim_{t \to \infty} g''(t) = 0$ and satisfies
	\[
	g''(f(x)) 
	= \frac{3 f(x)^2 \arccos \frac{1}{f(x)} - \sqrt{f(x)^2-1} - 2f(x)^2 \sqrt{f(x)^2 -1}}{f(x) (f(x)^2-1)^\frac{5}{2}}.
	\]
	Since $f(x) \geq 1$, we thus have
	\[
	|g''(f(x))|
	\leq \frac{3 \pi f(x)}{(f(x)^2-1)^\frac{5}{2}} 
	+ \frac{1}{{(f(x)^2-1)^2} }+ \frac{ 2f(x) }{(f(x)^2-1)^2} .
	\]
	Due to the monotonicity properties of $f$ there exists $x_0 >0$ such that
	$f(x) \geq \sqrt{2}$ for $0<x \leq x_0$. 
	
	In particular, for $x \leq x_0$ we have 
	$f(x)^2-1 \geq f(x)^2/2$ and thus
	\[
	|g''(f(x))|
	\leq C_1 \left( \frac{1}{f(x)^4} + \frac{1}{f(x)^3} \right)
	\]
	with a constant $C_1>0$ independent of $x$. Moreover, we note that Lemma~\ref{lemma:old_iota_bounds-uniform}
	and Lemma~\ref{lemma:mcmahon} imply
	\begin{equation}\label{lemma:iota-lower-bound}
	\iota(x) \geq \frac{j_{xk,k}}{k} - (1-\eps) \frac{\pi}{4k} 
	\geq \frac{j_{0,k}}{k} - (1-\eps) \frac{\pi}{4k} 
	\geq \frac{\pi}{2}
	\end{equation}
	for $x>0$ and sufficiently large $k$. 
	Therefore,
	\[
	\begin{aligned}
		\frac{|g''(f(x))|}{x^2} 
		& \leq \frac{C_1}{x^2 f(x)^2} \left( \frac{1}{f(x)^2} + \frac{1}{f(x)} \right)
		= \frac{C_1}{\iota(x)^2} \left( \frac{1}{f(x)^2} + \frac{1}{f(x)} \right) \\
		& \leq C_1 \left( \frac{1}{f(x)^2} + \frac{1}{f(x)} \right) =: \kappa_1(x)
	\end{aligned}
	\]
	and the monotonicity of $f$ yields that the right hand side is strictly increasing in $x$.
	
	On the other hand, for $x > x_0$ we then have $f(x) \leq \sqrt{2}$ and hence
	\[
	\begin{aligned}
		\frac{|g''(f(x))|}{x^2}
		& \leq C_2 \left( \frac{1}{(f(x)^2-1)^\frac{5}{2}} 
		+ \frac{1}{{(f(x)^2-1)^2} }+ \frac{ 1 }{(f(x)^2-1)^2} \right) =: \kappa_2(x)
	\end{aligned}
	\]
	with a constant $C_2>0$ independent of $x$.
	Again, the monotonicity of $f$ implies 
	that the right hand side is strictly increasing in $x$.	
	
	After choosing a suitable constant $\kappa_0>0$, the function
	\[
	\kappa(x) = \begin{cases}
		\kappa_1(x) , & 0<x<x_0, \\
		\kappa_2(x) + \kappa_0 , & x \leq x_0
	\end{cases}
	\]
	is increasing and satisfies $\frac{|g''(f(x))|}{x^2} \leq \kappa(x)$ for $x>0$. 
	This completes the proof.
\end{proof}
Similarly, \eqref{eq: F_k definition} motivates the following helpful estimate.
\begin{lemma} \label{lemma:sinh_analytic}
	The function $\psi: \R \to \R$ given by 
	\[
	\psi(x) \coloneqq \begin{cases}
		\frac{\sinh x}{x}, \qquad & \text{$x \neq 0$,} \\
		0 , \qquad & x=0,
	\end{cases}
	\]
	is analytic and satisfies
	\[
	\psi(x) = \sum_{n=0}^\infty \frac{x^{2n}}{(2n+1)!} 
	\]
	and
	\[
	\psi^{(k)}(x) = \sum_{n=\lceil k/2 \rceil}^\infty \
	\frac{x^{2n-k}}{(2n-k)!(2n+1)} 
	\]
	for $k \in \N_0$.
	In particular, this implies 
	\[
	|\psi^{(k)}(x)| \leq \frac{1}{k+1} e^{x}
	\]
	for $x>0$.
\end{lemma}

\begin{proof}
	This follows from the series expansion of $\sinh$.
\end{proof}
The next result specifies in which sense \eqref{eq: G definition} is the limit 
problem for \eqref{eq: F_k definition}.
\begin{lemma} \label{lemma:integral_taylor_expansion}
	Let $x,y>0$ be such that $\frac{x}{y}<1$, and consider the map $F: \R \to 
	\R$ given by
	\[
	F(s) \coloneqq \begin{cases} 2 y \int_0^\infty K_0 \left( t 2 y  
	\frac{\sinh(s t)}{s t} \right) e^{-2xt} \, dt \qquad & s \neq 0, \\
	 \int_0^\infty K_0 \left( t   
	\right) e^{-\frac{x}{y}t} \, dt & s = 0,
	\end{cases}
	\]
	where $K_0$ denotes the modified Bessel function of the second kind of 
	order $0$.
	Then $F$ is twice continuously differentiable and satisfies
	\[
	\begin{aligned} 
	F'(0) & = 0 \\
	F''(0) & = - \frac{1}{12 y^2}  \int_0^\infty t^3 K_1(t) e^{-\frac{x}{y} t}  \, dt .
	\end{aligned}
	\]
\end{lemma}
\begin{proof}
We consider the function $h: \R \times (0,\infty) \to \R$ given by 
\[
	h(s,t) \coloneqq 2y K_0\left(t2y \psi(st) \right) e^{-2x t}
\] 
where the function $\psi$ is given in Lemma~\ref{lemma:sinh_analytic}. 
Then $h$ is smooth on $\R \times (0,\infty)$ and
\[
\frac{\del h}{\del s} (s,t) = 2yK_0'\left(t2y \psi(st) \right) t2yt\psi'(st) 
e^{-2x t} = 
- 4y^2t^2 K_1\left(t2y \psi(st) \right) \psi'(st) e^{-2x t}
\]
and Lemma~\ref{lemma:sinh_analytic} hence implies
\[
\left|\frac{\del h}{\del s} (s,t) \right| \leq 2 y^2 t^2 e^{-2x t} 
e^{st} K_1\left(t2y \psi(st) \right) .
\]
Additionally, applying Lemma~\ref{lemma:modified_bessel} with $\gamma=1$ gives
\[
K_1\left(t2y \psi(st)  \right) 
\leq C_{1,1} \left(1+ \frac{1}{\left(t2y \psi(st) \right)^2} \right)  
{e^{-t2y \psi(st) }}
\]
and hence
\[
\begin{aligned} 
\left|\frac{\del h}{\del s} (s,t) \right| &
\leq C_{1,1} \left(2 y^2 t^2+ 
\frac{1}{2  \psi(st)^2} \right)  
\exp\left(-t2y \psi(st) -2xt+st\right) \\
& \leq  C_{1,1} \left(2 y^2 t^2+ 1 \right)  
\exp\left(-t \left(2y \psi(st) +2x-s\right)\right) .
\end{aligned} 
\]
For $s \in \R$ we note that there exists $t_0>0$ such that $ 2y \psi(st) -s >0$ 
holds for $t>t_0$. By continuity, there exists $\eps>0$ such that $ 2y \psi(st) 
-s +x >0$ holds for $s \in (s_0-\eps,s_0+\eps)$ and $t>t_0$. Setting 
$M_0 \coloneqq \min\{2y \psi(st) +2x-s + 2x: (s,t) \in [s_0-\eps,s_0+\eps] 
\times [0,1]\}$, it follows that
$\left|\frac{\del h}{\del s} (s,t) \right|$ is bounded by the integrable 
function
\[
t \mapsto \begin{cases} 
	C_{1,1} \left(2 y^2 t^2+ 1 \right)  
	\exp\left(-t M\right) \qquad & 0<t< t_0 \\
	 C_{1,1} \left(2 y^2 t^2+ 1 \right) \exp\left(-t x\right)  & t>t_0,
\end{cases}
\]
for $s \in (s_0-\eps,s_0+\eps)$.
Consequently, $F$ is differentiable and, in particular,
\[
F'(0)=\int_0^\infty \frac{\del h}{\del s} (0,t) \, ds = 0
\]
since $\psi'(0)=0$.
Similarly, noting that
\[
\begin{aligned}
	\frac{\del^2 h}{\del^2 s} (s,t) = & - 4y^2t^3 K_1\left(t2y \psi(st) \right) 
	\psi''(st) e^{-2x t} \\
	& - 8y^3 t^4 K_1'\left(t2y \psi(st) \right)
	(\psi'(st))^2 e^{-2x t} \\
	= & - 4y^2t^3 K_1\left(t2y \psi(st) \right) 
	\psi''(st) e^{-2x t} \\
	&+ 4y^3 t^4 \left(K_0\left(t2y \psi(st) \right) + 
	K_2\left(t2y \psi(st) \right) \right)
	(\psi'(st))^2 e^{-2x t} ,
\end{aligned}
\]
it follows that $\left|\frac{\del^2 h}{\del^2 s} (s,t) \right|$ can 
be bounded by an integrable function in a neighborhood of any $s \in \R$.
Consequently,
$$
F''(0)=\int_0^\infty \frac{\del^2 h}{\del^2 s} (0,t) \, ds 
$$
where, using that $\psi''(0)= \frac{1}{3}$ by Lemma~\ref{lemma:sinh_analytic}, 
we find that
$$
\frac{\del^2 h}{\del^2 s} (0,t) = - \frac{4}{3} y^2 t^3 K_1\left(t2y\right)  
e^{-2x t} 
$$
and hence
$$
F''(0)= - \frac{4}{3}y^2 \int_0^\infty t^3 K_1(t2y) e^{-2x t}  \, dt =  - 
\frac{1}{12 y^2}  \int_0^\infty t^3 K_1(t) e^{-\frac{x}{y} t}  \, dt 
$$
as claimed.
\end{proof}
Similarly, the following result characterizes the limit problem which will 
appear in the treatment of derivatives of Bessel function zeros.
\begin{lemma} \label{lemma:integral_taylor_expansion_2nd}
	Let $x>0$ be fixed, and consider the map $F: \R \to \R$ given by
	$$
	F(s) \coloneqq 
	\begin{cases} 
	\int_0^\infty K_0 (t) 
	\frac{\exp\left(-\frac{x}{s}\arsinh(ts)\right)}{\sqrt{1+t^2s^2}} \, dt 
	\qquad 
	&\text{for $x \neq 0$}, \\
	\int_0^\infty K_0 (t) 
	\exp(-xt) \, dt
	& \text{for $x=0$.}
	\end{cases}
	$$
	where $K_0$ denotes the modified Bessel function on the second kind of 
	order $0$.
	Then $F$ is twice continuously differentiable and satisfies
	$$
	\begin{aligned} 
		F'(0) & = 0 \\
		F''(0) &= x \int_0^\infty t^3 \exp(-xt) K_0(t)  \, dt - \int_0^\infty 
		t \exp(-xt) K_0(t)  \, dt .
	\end{aligned}
	$$
	Moreover, $F''$ is bounded on $\R$.
\end{lemma}
\begin{proof}
	We consider the function $h: \R \times (0,\infty) \to \R$ given by 
	$$
	h(s,t) \coloneqq  
	\begin{cases} 
	\frac{\exp\left(-\frac{x}{s}\arsinh(ts)\right)}{\sqrt{1+t^2s^2}} \qquad 
	&\text{for $x \neq 0$}, \\
	\exp\left(-xt\right) & \text{for $x=0$.}
	\end{cases}
	$$ 
	It is easy to see that $h$ is continuous and $|h|\leq 1$.
	Then, for $s \neq 0$,
	$$
	\frac{\del h}{\del s} (s,t) =  h(s,t)
	\left(\frac{x}{s^2} \arsinh(ts)- \frac{tx}{s\sqrt{1+t^2 s^2}} -\frac{ t^2 
	s}{1+t^2s^2}\right)
	$$
	where it follows from L'Hôpital's rule that
	$$
	\lim_{s \to 0} \left(\frac{\arsinh(ts)}{s^2} -\frac{t}{s\sqrt{1+t^2 s^2}} 
	\right)
	= \lim_{s \to 0}\frac{\arsinh(ts)\sqrt{1+t^2 s^2}-ts}{s^2\sqrt{1+t^2 s^2}}
	= 0
	$$
	for $t \geq 0$. Consequently, $\frac{\del h}{\del s}$ can be extended to a 
	continuous function on 
	$\R \times [0,\infty)$ by setting
	$$
	\frac{\del h}{\del s} (0,t) = 0.
	$$
	Since $0 \leq h \leq 1$, it further follows that 
	$$
	M_0 \coloneqq \sup \left\{ \left|\frac{\del h}{\del s} (s,t) \right| 
	e^{-\frac{t}{2}}: (s,t) 
	\in \R \times (0,\infty) \right\} < \infty .
	$$
	Using Lemma~\ref{lemma:modified_bessel}(i) we thus find that
	$$
	\left|K_0 (t) \frac{\del h}{\del s} (s,t) \right|
	 \leq C_{0,\frac{1}{2}} \left(1 + \frac{1}{\sqrt{t}} \right) e^{-t} 
	 \left|\frac{\del h}{\del s} (s,t) \right| 
	\leq M_0 C_{0,\frac{1}{2}} \left(1 + \frac{1}{\sqrt{t}} \right) 
	e^{-\frac{t}{2}} 
	$$
	and since the right hand side is integrable and independent of $s$, it 
	follows that $F$ is differentiable and 
	$$
	F'(0)=\int_0^\infty K_0 (t)  \frac{\del h}{\del s} (0,t) \, ds = 0 .
	$$
	Similarly, we note that
	$$
	\begin{aligned}
		\frac{\del^2 h}{\del^2 s} (s,t) = & \left(\frac{\del h}{\del 
		s}(s,t)\right)^2 (h(s,t))^{-1} \\
		&+ h(s,t)\left(\frac{2xt}{s^2\sqrt{1+t^2 s^2}} + \frac{xt^3}{(1+t^2 
		s^2)^\frac{3}{2}} - \frac{2x \arsinh(ts)}{s^3} - \frac{t^2}{1+t^2 s^2} 
		+ \frac{2 t^4 s^2}{(1+t^2 s^2)^2} \right) 
	\end{aligned}
	$$
	so, using L'H\^opital's rule again we find that
	$$
	\lim_{s \to 0}\left(\frac{t}{s^2\sqrt{1+t^2 s^2}}  - \frac{ 
	\arsinh(ts)}{s^3}\right)
	= \lim_{s \to 0}\frac{ts - \arsinh(ts) \sqrt{1+t^2 s^2}}{s^3\sqrt{1+t^2 
	s^2}}
	= 0
	$$
	and hence $\frac{\del^2 h}{\del^2 s}$ can be extended to a 
	continuous function on 
	$\R \times [0,\infty)$ by setting
	$$
	\frac{\del^2 h}{\del^2 s} (0,t) = h(0,t) (xt^3 - t^2) = \exp(-xt)(xt^3 - 
	t^2) .
	$$
	Moreover,
	$$
	M_1 \coloneqq \sup \left\{ \left|\frac{\del^2 h}{\del^2 s} (s,t) \right| 
	e^{-\frac{t}{2}}: (s,t) 
	\in \R \times (0,\infty) \right\} < \infty 
	$$
	and we thus have 
	$$
	\left|K_0 (t) \frac{\del^2 h}{\del^2 s} (s,t) \right|
	\leq C_{0,\frac{1}{2}} \left(1 + \frac{1}{\sqrt{t}} \right) e^{-t} 
	\left|\frac{\del h}{\del s} (s,t) \right| 
	\leq M_0 C_{0,\frac{1}{2}} \left(1 + \frac{1}{\sqrt{t}} \right) 
	e^{-\frac{t}{2}} 
	$$
	by Lemma~\ref{lemma:modified_bessel}(i). We thus conclude that $F'$ is differentiable and 
	bounded since
	$$
	|F''(s)| \leq \int_0^\infty M_0 C_{0,\frac{1}{2}} \left(1 + 
	\frac{1}{\sqrt{t}} \right) 
	e^{-\frac{t}{2}} \, dt < \infty,
	$$
	and satisfies
	$$
	\begin{aligned} 
	F''(0)&=\int_0^\infty K_0 (t)  \frac{\del^2 h}{\del^2 s} (0,t) \, ds 
	= x \int_0^\infty\exp(-xt) t^3 K_0(t)  \, dt - \int_0^\infty 
	\exp(-xt) t K_0(t)  \, dt .
	\end{aligned}
	$$
        This completes the proof.
      \end{proof}
      
%
%
%
%
%

%

%
%%%%%%%%%%%%%%%%%%%%%%%%%%%%%%%%%%%%%%%%%%%%%%%%%%%%%%%%%%%%%%%%%%%%%%%%%%%%%%%%%%%%%%%%%%%%%%%%%%%%%%%%%%%%%
%
\section{Improved estimates for zeros of Bessel functions}
\label{Section: New estimates}
The goal of this section is to prove the asymptotic expansion stated in
Theorem~\ref{theorem:intro_second_order_expansion}.
The basic strategy consists of starting with the known convergence estimates, 
and using Taylor expansion to iteratively improve the order.

\subsection{First order coefficients}
We first improve the first order estimates given in
Lemma~ \ref{lemma:old_iota_bounds-uniform}.
\begin{lemma} \label{lemma:new_iota_bounds}
	Let 
	$x>0$
	Then there exists a constant $C_x>0$ and $k_0 \in \N$, such that
	\[
	- \frac{\pi}{4k} 	
	\frac{f(x)}{\sqrt{f(x)^2-1}}-\frac{C_x}{k^2} 
	\leq \frac{j_{xk,k}}{k} - \iota(x) 
	\leq - \frac{\pi}{4k} 	
	\frac{f(x)}{\sqrt{f(x)^2-1}} + \frac{C_x}{k^2}
	\]
	holds for $k \geq k_0$.
\end{lemma}
Note that the estimate can also be written as 
\[
- \frac{\pi}{4k} 	
\frac{\iota(x)}{\sqrt{\iota(x)^2-x^2}}-\frac{C}{k^2} 
\leq \frac{j_{xk,k}}{k} - \iota(x) 
\leq - \frac{\pi}{4k} 	
\frac{\iota(x)}{\sqrt{\iota(x)^2-x^2}} + \frac{C}{k^2} .
\]
\begin{proof}
	In the following, we write $\iota = \iota(x)$, $\iota_k = \iota_k(x)$ since 
	$x$ is always fixed.
	
	Recall that we set $\iota_k(x)=  \frac{j_{x k,k}}{k}$ and the functions 
	satisfy
	\begin{align*}
		\frac{d}{dx} \iota_k & = F\left(\iota_k,x,\frac{1}{k}\right) \\
		\frac{d}{dx} \iota & = F(\iota,x,0)
	\end{align*}
	in $(-1,\infty)$ with $F$ defined by
	\[
	F(y,x,s) \coloneqq 2 y \int_0^\infty K_0 \left( -t 2 y  
	\frac{\sinh(s t)}{s t} \right) e^{-2xt} \, dt .
	\] 
	Now consider 
	$u_k(x)\coloneqq\iota_k(x)-\iota(x)$ so that
	\[
	\frac{d}{dx} u_k = 
	\frac{F\left(\iota_k,x,\frac{1}{k}\right)-F\left(\iota,x,0\right)}{\iota_k(x)-\iota(x)}
	u_k(x) = \beta_k(x) u_k(x) 
	\]
	where we set 
	\[
	\beta_k(x)\coloneqq\frac{F\left(\iota_k,x,\frac{1}{k}\right)-F\left(\iota,x,0\right)}{\iota_k(x)-\iota(x)}
	 .
	\]
	Note that $\beta_k$ is well-defined by \eqref{eq: iota - iota_k inequality}. 
	In particular, we find that
	\[
	u_k(x)= u_k(0) \exp \left( \int_0^x \beta_k(t) \, dt \right) .
	\]
	In the following, we estimate the different parts of this formula.
	Firstly, we consider $u_k(0)$ and recall that $\iota(0)=\pi$ and therefore
	\[
	u_k(0)=  \frac{j_{0,k}}{k} - \pi ,
	\]
	so Lemma~\ref{lemma:mcmahon} yields
	\begin{equation}  \label{eq: u_k(0) estimates}
	 u_k(0) = - \frac{\pi}{4 k} + \frac{1}{8 \pi k 
		(k-\frac{1}{4})} + O \left(\frac{1}{k^3}\right) .
	\end{equation}
	Next, we give several estimates for $F$ and first note that 
	Lemma~\ref{lemma:integral_taylor_expansion} gives
	\begin{equation} \label{eq:F_k_estimate}
	\begin{aligned}
		F\left(y,x,\frac{1}{k}\right)-F\left(y,x,0\right) = 
		R_k\left(y,x\right)
	\end{aligned}
	\end{equation}
	where we have set
	\[
	R_k\left(y,x\right) \coloneqq 
	\int_0^\frac{1}{k} \frac{\del^2 F}{\del^2 s}(y,x,s) \left(\frac{1}{k} 
	-s\right) \,ds ,
	\]
	so, in particular, 
	\[
	|R_k\left(y,x\right)| \leq \frac{1}{k^2} \max_{s \in \left[0, 
	\frac{1}{k}\right]} \left|\frac{\del^2 F}{\del^2 s}(y,x,s) \right| .
	\]
	Note that the continuity of 
	$\frac{\del^2 F}{\del^2 s}$ stated in 
	Lemma~\ref{lemma:integral_taylor_expansion} implies 
	\begin{equation} \label{eq:M_0_def} 
	M_x \coloneqq \max \left\{ 
	\left|\frac{\del^2 F}{\del^2 s}(y,x',s) \right|: (y,x',s) \in 
	[0, \iota + 1] 
	\times [0,x] \times 
	[0,1]  \right\} < \infty .
	\end{equation}
 	Moreover, Lemma~\ref{lemma:old_iota_bounds-uniform} implies there exists $k_0 \in 
 	\N$ such that  	$|\iota(x')-\iota_k(x')|< 1$ for $k \geq k_0$ and $0<x'\leq x$.
 	Consequently, we have
	$\iota_k(x') \leq \iota(x') + | \iota_k(x') - \iota(x')| \leq \iota(x) + 1$ for $k \geq k_0$, $0<x'\leq x$
	and hence
	\[
	R_k\left(\iota_k,x\right) \leq \frac{M_x}{k^2}
	\]
	for $k \geq k_0$.
	
	Similarly, noting that $F(y,x,0)=g\left(\frac{y}{x}\right)$, Taylor's 
	theorem 
	gives 
	\begin{equation}\label{eq:G_estimate}
	F\left(\iota_k,x,0\right)  - 
	F\left(\iota,x,0\right) = g\left(\frac{\iota_k}{x}\right) - 
	g\left(\frac{\iota}{x}\right) = 	
	\frac{g'\left(\frac{\iota}{x}\right)}{x} 
	(\iota_k(x) - \iota(x)) + Q_k(x)
	\end{equation}
	where
	\[
	Q_k(x) = \int_{\frac{\iota}{x}}^\frac{\iota_k}{x} g''(t) 
	\left(\frac{\iota_k}{x}-t\right) \, dt,
	\]
	so, in particular,
	\begin{equation*}
	|Q_k(x)| \leq 
	\left| g''\left(\frac{\iota}{x}\right) \right| \left(\frac{\iota}{x}-\frac{\iota_k}{x}\right)^2 
	= \frac{\left| g''\left(\frac{\iota}{x}\right) \right|}{x^2} ( \iota - \iota_k)^2 .
	\end{equation*}
	Since $g''\left(\frac{\iota}{x}\right) = g(f(x))$, Lemma~\ref{lemma:g-f-integrability} yields
	\begin{equation} \label{eq:Q_k-estimate}
	\frac{\left| g''\left(\frac{\iota}{x}\right) \right|}{x^2} \leq \kappa(x)
	\end{equation}
	with an increasing function $\kappa:(0,\infty) \to \R$.
	
	Overall, combining \eqref{eq:F_k_estimate} and \eqref{eq:G_estimate} then 
	yields
	\[
	\begin{aligned} 
	F\left(\iota_k,x,\frac{1}{k}\right)-F\left(\iota,x,0\right)
	& = 
	F\left(\iota_k,x,\frac{1}{k}\right)- F\left(\iota_k,x,0\right) + 
	F\left(\iota_k,x,0\right)  - 
	F\left(\iota,x,0\right) \\
	& = \frac{g'\left(\frac{\iota}{x}\right)}{x} 
	(\iota_k(x) - \iota(x)) + Q_k(x) + R_k(\iota_k,x)
	\end{aligned}
	\]
	and it follows that
	\begin{align*}	
		\beta_k(x) & 
		=\frac{F\left(\iota_k,x,\frac{1}{k}\right)-F\left(\iota,x,0\right)}{\iota_k(x)-\iota(x)}
		 \\
		& = \frac{g'\left(\frac{\iota}{x}\right)}{x} 
		 + \frac{Q_k(x) + R_k(\iota_k,x) }{\iota_k(x)-\iota(x)} .
	\end{align*}
	Setting $\beta_0(x)\coloneqq \frac{g'\left(\frac{\iota}{x}\right)}{x} $ we 
	thus have
	\[
	\beta_k(x) =
	\beta_0(x) +  \Theta_k(x)
	\]
	with 
	$$
	 \Theta_k(x) \coloneqq \frac{Q_k(x) + 
		R_k(\iota_k,x) }{\iota_k(x)-\iota(x)}.
	$$
	Using Lemma~\ref{lemma:old_iota_bounds-uniform} and \eqref{eq:Q_k-estimate}, 
	after possibly enlarging $k_0$, we find that
	\begin{equation}\label{eq:first_Theta_bound}
	\begin{aligned} 
	|\Theta_k(x) | & \leq 
	\kappa(x) ( \iota - \iota_k)  + 	
	\frac{M_x}{k^2 \left(\iota-\iota_k\right)}  \\
	& \leq \kappa(x) \frac{\pi \exp(x)}{4k} + \frac{M_x}{k \frac{\pi}{8}} 
	\end{aligned}
	\end{equation}
	for $k \geq k_0$.
	Then 
	\begin{equation} \label{eq:exp_int_with_error}
	\exp \left(\int_0^x \beta_k(t) \, dt\right) 
	= \exp \left(\int_0^x 
	\beta_0(t) \,dt \right) + \exp \left(\int_0^x 
	\beta_0(t) \,dt \right) \left(\exp \left( \int_0^x\Theta_k(t) \,dt \right) 
	-1 \right) .
	\end{equation}
	Next, we further simplify $\exp \left(\int_0^x 
	\beta_0(t) \,dt \right)$ and note that using the substitution $t=f^{-1}(s)$ 
	with $f$ given as 
	above, we may write
	$$
	\begin{aligned}
		\int_0^x \beta_0(t) \, dt = \int_0^x 
		\frac{g'\left(\frac{\iota(t)}{t}\right)}{t} \, dt 
		= \int_{f(x)}^\infty g'(s) \frac{(|f^{-1})'(s)|}{f^{-1}(s)} \, ds .
	\end{aligned}
	$$
	Using the identity $\arccos + \arcsin = \frac{\pi}{2}$, we further find
	$$
	\begin{aligned}
		\frac{|(f^{-1})'(s)|}{f^{-1}(s)}   
		=\frac{\frac{\sqrt{s^2-1}}{s}}{\sqrt{s^2- 1} - 
			\arccos \frac{1}{s} },
	\end{aligned}
	$$
	while
	$$
	\begin{aligned}
		g'(s)
		= \frac{\sqrt{s^2-1} - \arccos 
		\frac{1}{s}}{\left(s^2-1\right)^\frac{3}{2}}
	\end{aligned}
	$$
	by Lemma~\ref{lemma:g_properties}, 
	and therefore
	$$
	g'(s) \frac{(|f^{-1})'(s)|}{f^{-1}(s)}  
	= \frac{1}{s(s^2-1)} .
	$$
	Since this term has the antiderivative
	$$
	G(s)= \frac{1}{2}\ln(s^2-1) - \ln(s) 
	$$
	and $\lim_{s \to \infty} G(s) = 0$, 
	it follows that
	$$
	\int_0^x \beta_0(t) \, dt = G(t) \big|_{f(x)}^\infty =  \ln(f(x)) - 
	\frac{1}{2} \ln(f(x)^2-1) 
	$$
	and hence
	$$
	\exp\left(\int_0^x\beta_0(t) \, dt \right)= \frac{f(x)}{\sqrt{f(x)^2-1}} .
	$$
	Recalling \eqref{eq:exp_int_with_error}, this gives
	$$
	\exp \left(\int_0^x \beta_k(t) \,dt\right)
	 = \frac{f(x)}{\sqrt{f(x)^2-1}} 
	 +  \frac{f(x)}{\sqrt{f(x)^2-1}}
	\left(\exp \left( \int_0^x\Theta_k(t) \,dt \right) 
	-1 \right) .
	$$
	Note that 
	\begin{equation} \label{eq:M_x_monotone_statement}
	\kappa(t) \leq \kappa(x), \qquad M_t \leq M_x \qquad \text{for $0<t<x$}
	\end{equation}
	and hence \eqref{eq:first_Theta_bound} gives
	\begin{equation} \label{eq:rough_error_estimate}
		\begin{aligned} 
	\left| \exp \left( \int_0^x\Theta_k(t) \,dt \right) -1 \right|
	& \leq \max_{x' \in [0,x]} \exp \left( \int_0^{x'}\Theta_k(t) \,dt \right) 
	\int_0^{x'}\Theta_k(t) \,dt \\
	& \leq \exp \left( x \left(\frac{8 M_x}{\pi k} 
	+ \kappa(x) \frac{\pi \exp(x)}{4k}\right) \right)
	 x \left( \frac{8 M_x}{\pi k} 
	+ \kappa(x) \frac{\pi \exp(x)}{4k} \right) \\
	& \leq \frac{\tilde C_x}{k}
	\end{aligned}
	\end{equation}
	for $k \geq k_0$, where we have set
	$$
	\tilde C_x \coloneqq 2 x \left(  \frac{8 M_x}{\pi } 
	+ \kappa(x) \frac{\pi \exp(x)}{4} \right) .
	$$
	Recalling 
	$$
	u_k(x)= u_k(0) \exp \left( \int_0^x \beta_k(t) \, dt \right), 
	$$
	we thus find that
	\begin{equation}\label{eq:u_k-formula}
	u_k(x) = u_k(0)  \frac{f(x)}{\sqrt{f(x)^2-1}}  
	+ u_k(0)  \frac{f(x)}{\sqrt{f(x)^2-1}}
	\left(\exp \left( \int_0^x\Theta_k(t) \,dt \right) 
	-1 \right)  .
	\end{equation}
	Finally, $u_k(0) = - \frac{\pi}{4 k} + \frac{1}{8 \pi k 
	(k-\frac{1}{4})} + O \left(\frac{1}{k^3}\right)$ and the 
	estimate
	$$
	\left|  \frac{f(x)}{\sqrt{f(x)^2-1}}
	\left(\exp \left( \int_0^x\Theta_k(t) \,dt \right) 
	-1 \right) \right|
	\leq \frac{f(x)}{\sqrt{f(x)^2-1}} \frac{\tilde C_x}{k^2} 
	$$
	for $k \geq k_0$ yield the claim with
	\[
	C_x \coloneqq \frac{1}{8\pi} +  \frac{f(x)}{\sqrt{f(x)^2-1}}\tilde C_x . 
	\]
\end{proof}

\begin{remark}
	From the preceding proof, it follows that the constant $C_x$ in 
	Lemma~\ref{lemma:new_iota_bounds} satisfies
	$$
	C_x \leq \left(\frac{1}{8 \pi} + \frac{\pi}{4} 2 x \left(  \frac{8 M_x}{\pi } 
	+ \kappa(x) \frac{\pi \exp(x)}{4} \right) \right)\frac{f(x)}{\sqrt{f(x)^2-1}} .
	$$
	Notably, by \eqref{eq:M_x_monotone_statement} and since the function $f$ is decreasing, we also have
	%and $\lim_{t \to 0} f(t)=\infty$
	\begin{equation} \label{eq:C_x_estimate_for_smaller_x}
	C_{x'} \leq \left(\frac{1}{8 \pi} + \frac{\pi}{4} 2 x \left(  \frac{8 M_x}{\pi } 
	+ \kappa(x) \frac{\pi \exp(x)}{4} \right) \right)\frac{f(x)}{\sqrt{f(x)^2-1}} 
	\qquad \text{for $0<x'<x$.}
	\end{equation}
\end{remark}

\subsection{Second order coefficients}
We now use the information from Lemma~\ref{lemma:new_iota_bounds} to 
characterize the next term in the asymptotic expansion of ${j_{xk,k}}/{k}$ as 
$k \to \infty$.

\begin{theorem}\label{theorem:second_order_expansion}
	Let 
	$x>0$
	Then there exists $\zeta_x \in \R $ such that
	\begin{equation}\label{eq:second_order_theorem_statement}
	\frac{j_{xk,k}}{k} - \iota(x) = - \frac{\pi}{4k} 	
	\frac{f(x)}{\sqrt{f(x)^2-1}} + \frac{\zeta_x}{k^2} + 
	o\left(\frac{1}{k^2}\right)
	\end{equation}
	as $k \to \infty$. Moreover, we have
	\begin{equation} \label{eq:zeta_x_first_def} 
          \zeta_x \coloneqq \frac{\pi}{4}\frac{f(x)}{\sqrt{f(x)^2-1}}
          \left( \frac{1}{2\pi^2}  - \int_0^x \Theta_0(t) \, dt \right)
	\end{equation} 
	where, writing $\iota = \iota(x)$,
	\begin{equation} \label{eq:Theta_0_def}
	\Theta_0(x) \coloneqq -\frac{\pi}{8x^2} g''(f(x)) 
	\frac{f(x)}{\sqrt{f(x)^2-1}} -
	\frac{\del^2 F}{\del^2 s}(\iota,x,0) \frac{2}{\pi} 
	\frac{\sqrt{f(x)^2-1}}{f(x)}.
	\end{equation}
	Here, we have
	\[
	\begin{aligned} 
	\frac{\del^2 F}{\del^2 s}(\iota,x,0) & = 
	 - \frac{1}{12 \iota^2} \int_0^\infty t^3 K_1(t) e^{-\frac{x}{\iota} t}  \, 
	 dt
	 = - \frac{1}{12 x^2 f(x)^2} \int_0^\infty t^3 K_1(t) e^{-\frac{t}{f(x)} }  
	 \, 
	 dt
	\end{aligned} 
	\]
	and $g''$ is given by \eqref{eq:g-second-derivative-def}.
\end{theorem}
\begin{proof}
	As in the previous proof, we set $\iota_k(x)= \frac{j_{x k,k}}{k}$ and 
	$u_k(x)=\iota_k(x)-\iota(x)$.	
	Recall that by \eqref{eq:u_k-formula}, 
	\begin{equation} \label{eq:expansion_error_term} 
	u_k(x) = u_k(0)  \frac{f(x)}{\sqrt{f(x)^2-1}}  
	+ u_k(0)  \frac{f(x)}{\sqrt{f(x)^2-1}}
	\left(\exp \left( \int_0^x\Theta_k(t) \,dt \right) 
	-1 \right)  
	\end{equation}
	with $\Theta_k(x)$ given by
	$$
	\Theta_k(x) = \frac{Q_k(x) + 
		R_k(\iota_k,x) }{\iota_k(x)-\iota(x)}
	$$
	where
	$$
	\begin{aligned}
			Q_k(x) & = \int_{\frac{\iota}{x}}^\frac{\iota_k}{x} g''(t) 
		\left(\frac{\iota_k}{x}-t\right) \, dt, \\
			R_k\left(y,x\right) &= 
		\int_0^\frac{1}{k} \frac{\del^2 F}{\del^2 s}(y,x,s) \left(\frac{1}{k} 
		-s\right) \,ds .
	\end{aligned}
	$$
	By McMahon's expansion (see Lemma~\ref{lemma:mcmahon} and Remark~\ref{remark:on-mcmahon}), 
	the first term in \eqref{eq:expansion_error_term} satisfies
	$$
	u_k(0)  \frac{f(x)}{\sqrt{f(x)^2-1}}  = - \frac{\pi}{4k} 	
	\frac{f(x)}{\sqrt{f(x)^2-1}} + \frac{1}{8 \pi k^2}  
	\frac{f(x)}{\sqrt{f(x)^2-1}} + O(k^{-3}) .
	$$
	Similarly, \eqref{eq:rough_error_estimate} gives
	$$
	\begin{aligned}
	& u_k(0)  \frac{f(x)}{\sqrt{f(x)^2-1}}
	\left(\exp \left( \int_0^x\Theta_k(t) \,dt \right) 
	-1 \right) \\
	& = \left(- \frac{\pi}{4 k} + \frac{1}{8 \pi k^2 
		}\right)  \left(\exp \left( 
		\int_0^x\Theta_k(t) \,dt \right) 
	-1 \right) + O(k^{-3}) .
	\end{aligned}
	$$
	In order to prove the claim, we write
        $$
        \exp \left( \int_0^x\Theta_k(t) \,dt \right)
          = \exp \left( \frac{1}{k} \int_0^x\Theta_0(t) + (k\Theta_k(t) - \Theta_0(t)) \,dt \right)
          $$
          and apply Taylor's theorem to the function 
          \[
          s \mapsto \exp \left( s \int_0^x\Theta_0(t) + (k\Theta_k(t) - 
          \Theta_0(t)) \,dt \right)
          \]
          in \(s=0\), yielding
          \begin{equation}\label{eq:exp-1-rewrite}
          \begin{aligned}
         &  \exp \left(\int_0^x\Theta_k(t) \,dt \right) -1 \\
         = & \frac{1}{k} \int_0^x\Theta_0(t) + (k\Theta_k(t) - \Theta_0(t)) 
         \,dt \\
         & + \frac{1}{2 k^2} \left(\int_0^x\Theta_0(t) + (k\Theta_k(t) - 
         \Theta_0(t)) \,dt \right)^2
          \exp \left( \xi \int_0^x\Theta_0(t) + (k\Theta_k(t) - \Theta_0(t)) \,dt\right)
           \end{aligned}
           \end{equation}
           for some \(\xi \in (0,\frac{1}{k})\).

           \textbf{\namedlabel{Claim1}{Claim 1}}: \emph{For \(t>0\), it holds 
           that
           $$
           k\Theta_k(t) - \Theta_0(t) \to 0
           $$
           as \(k \to \infty\). Moreover, there exists \(h \in L^1((0,x))\) 
           such that
           \(|k\Theta_k(t) - \Theta_0(t)| \leq h\) holds for \(t \in (0,x)\) 
           and sufficiently large \(k\).
         }

         In view of \eqref{eq:exp-1-rewrite} and Lebesgue's theorem, the claim completes
         the proof. It thus remains to prove \ref{Claim1}.
         To this end, we write
         \(k\Theta_k(t) =\frac{ k R_k(\iota_k,x)}{\iota_k(x)-\iota(x)}
         + 	\frac{ k Q_k(\iota_k,x)}{\iota_k(x)-\iota(x)}\)
         and estimate both terms.
        
	Firstly, we first note that
		$$
	\frac{ k R_k(\iota_k,x)}{\iota_k(x)-\iota(x)} =  \frac{k 
	}{\iota_k(x)-\iota(x)}
	\int_0^\frac{1}{k} \frac{\del^2 F}{\del^2 s}(\iota_k,x,s) \left(\frac{1}{k} 
	-s\right) \,ds,
	$$
	where
	$$
	\begin{aligned} 
	&\left| \int_0^\frac{1}{k} \frac{\del^2 F}{\del^2 s}(\iota_k,x,s) \left(\frac{1}{k} 
	-s\right) \,ds - \frac{1}{2k^2} \frac{\del^2 F}{\del^2 s}(\iota,x,0) \right|  
	\leq \frac{A_x^k}{k^2}
	\end{aligned}
	$$
	with 
	$$
	A_x^k \coloneqq  \max\left\{\left| \frac{\del^2 F}{\del^2 s}(\iota_k,t,s)
            - \frac{\del^2 F}{\del^2 s}(\iota,t,0)\right|: 0 \leq s \leq \frac{1}{k},
        0 \leq t \leq x \right\}
	$$
	for $k \geq k_0$.
	Moreover, by Lemma~\ref{lemma:new_iota_bounds} we have
	\begin{equation}\label{eq:k-pre-factor-estimate}
	\left|\frac{k }{\iota_k(x)-\iota(x)} + \frac{4k^2}{\pi} \frac{\sqrt{f(x)^2-1}}{f(x)}\right| 
	\leq \frac{4C_x}{\pi}
      \end{equation}
      with $C_x$ given by \eqref{eq:C_x_estimate_for_smaller_x}
	and thus
	\begin{equation}\label{eq:R-Terms-estimate}
		\left|\frac{ k R_k(\iota_k,x)}{\iota_k(x)-\iota(x)} + \frac{2}{\pi} 
		\frac{\sqrt{f(x)^2-1}}{f(x)} \frac{\del^2 F}{\del^2 s}(\iota,x,0) \right| 
		\leq \frac{4 C_x}{\pi k^2} M_x + \frac{4}{\pi} 
		\frac{\sqrt{f(x)^2-1}}{f(x)} A_x^k .
	\end{equation}
	with $M_x>0$ given by \eqref{eq:M_0_def}. 
	Similarly, we write 
	$$
	\begin{aligned} 
	&\left| \frac{ k Q_k(x)}{\iota_k(x)-\iota(x)}  + \frac{\pi}{8x^2} \frac{f(x)}{\sqrt{f(x)^2-1}} 
	g''\left(f(x)\right)\right| \\
	 \leq & \left|\frac{k }{\iota_k(x)-\iota(x)} + \frac{4k^2}{\pi} \frac{\sqrt{f(x)^2-1}}{f(x)}\right|  \left|Q_k(x)\right| \\
          & + \left|\frac{4k^2}{\pi} \frac{\sqrt{f(x)^2-1}}{f(x)}\right| 
          \left|Q_k(x) + \frac{g''\left(f(x)\right)}{2}\left(\frac{\iota_k}{x} - \frac{\iota}{x}\right)^2\right|
	\\
	& + \left|\frac{4k^2}{\pi} \frac{\sqrt{f(x)^2-1}}{f(x)}\right| \left|\frac{g''\left(f(x)\right)}{2}\left(\frac{\iota_k}{x} - \frac{\iota}{x}\right)^2 
	- \frac{\pi^2}{16 x^2 k^2} \frac{f(x)^2}{f(x)^2 - 1} \frac{g''\left(f(x)\right)}{2}\right|
        \end{aligned}
        $$
        where we recall 
	$$
	\frac{ k Q_k(x)}{\iota_k(x)-\iota(x)} 
	= \frac{k }{\iota_k(x)-\iota(x)}  \int_{\frac{\iota}{x}}^\frac{\iota_k}{x} 
	g''(t) 
	\left(\frac{\iota_k}{x}-t\right) \, dt .
	$$
        We use \eqref{eq:k-pre-factor-estimate} to estimate
        $$
        \begin{aligned}
        \left|\frac{k }{\iota_k(x)-\iota(x)}
          + \frac{4k^2}{\pi} \frac{\sqrt{f(x)^2-1}}{f(x)}\right|  \left|Q_k(x)\right|
          & \leq \frac{4C_x}{\pi} \frac{|g''(f(x))|}{2x^2}(\iota_k-\iota)^2 \\
          & \leq \frac{4C_x}{\pi} \frac{\pi^2}{16k^2}\exp(x)  \frac{|g''(f(x))|}{x^2} .
          \end{aligned} 
          $$
        Moreover, by Lemma~\ref{lemma:old_iota_bounds-uniform},
        $$
        \begin{aligned}
        \left|Q_k(x)  + \frac{g''\left(f(x)\right)}{2}\left(\frac{\iota_k}{x}
            - \frac{\iota}{x}\right)^2\right|
          &  \leq  \frac{|g''(f(x))-g''\left(\frac{\iota_k}{x}\right)|}{2x^2} (\iota_k-\iota)^2 \\
          & \leq \frac{\pi^2}{16 k^2} \exp(x) 
          \frac{|g''(f(x))-g''\left(\frac{\iota_k}{x}\right)|}{2x^2}
        \end{aligned}
        $$
        and Lemma~\ref{lemma:new_iota_bounds} gives
        $$
        \begin{aligned}
          \left|\frac{g''\left(f(x)\right)}{2}\left(\frac{\iota_k}{x}
          - \frac{\iota}{x}\right)^2 - \frac{\pi^2}{16x^2 k^2}\frac{f(x)^2}{f(x)^2 - 1}
          \frac{g''\left(f(x)\right)}{2}\right|
          & \leq \frac{|g''(f(x))|}{2x^2} \frac{f(x)}{\sqrt{f(x)^2-1}}  \frac{\pi C_x}{k^3} .
          \end{aligned}
          $$
          Overall, it follows that
          \begin{equation}\label{eq:Q-Terms-estimate}
          \begin{aligned}
            &\left| \frac{ k Q_k(x)}{\iota_k(x)-\iota(x)}
              + \frac{\pi}{8x^2} \frac{f(x)}{\sqrt{f(x)^2-1}} g''\left(f(x)\right)\right| \\
            & \leq 
              \frac{c_0 C_x}{k^2} \frac{|g''(f(x))|}{x^2}
              + c_0 \exp(x) \frac{|g''(f(x))-g''\left(\frac{\iota_k}{x}\right)|}{x^2}
              + \frac{c_0 C_x}{k} \frac{|g''(f(x))|}{x^2}
          \end{aligned}
          \end{equation}
          for some \(c_0 >0\) independent of \(k\) and \(x\).
	Recalling that $\Theta_0$ is given by \eqref{eq:Theta_0_def}, \eqref{eq:R-Terms-estimate} 
	and \eqref{eq:Q-Terms-estimate} then yield
	\begin{equation}
        \begin{aligned}
           \left| k \Theta_k(x) - \Theta_0(x) \right|
          &  \leq \frac{c_04 C_x}{ k^2} M_x +c_0 A_x^k
            + \frac{c_0 C_x}{k^2} \frac{|g''(f(x))|}{x^2} \\
          &    + c_0 \exp(x) \frac{|g''(f(x))-g''\left(\frac{\iota_k}{x}\right)|}{x^2}
              + \frac{c_0 C_x}{k} \frac{|g''(f(x))|}{x^2}
        \end{aligned}
        \end{equation}
        after possibly enlarging \(c_0\). Notably, the right hand side tends to zero
        pointwisely as \(k \to \infty\) by Lemma~\ref{lemma:new_iota_bounds}.
        In particular, this proves the first part of \ref{Claim1}.
        
        Recalling that \(A_x^k, M_x, C_x\) are increasing in \(x\), it further follows that for
        \(t \in (0,x)\), 
        	$$
        \begin{aligned}
           \left| k \Theta_k(t) - \Theta_0(t) \right|
          &  \leq \frac{c_04 C_x}{ k^2} M_x +c_0 A_x^k
            + \frac{c_0 C_x}{k^2} \frac{|g''(f(t))|}{t^2} \\
          &    + c_0 \exp(x) \frac{|g''(f(t))|}{t^2}
            + c_0 \exp(x)\frac{|g''\left(\frac{\iota_k(t)}{t}\right)|}{t^2}
              + \frac{c_0 C_x}{k} \frac{|g''(f(t))|}{t^2}
        \end{aligned}
        $$
        and the right hand side defines an integrable function on \((0,x)\)
        by Lemma~\ref{lemma:g-f-integrability}.
        Indeed, note that by Lemma~\ref{lemma:old_iota_bounds-uniform},
        \[
        {\iota_k(t)} \geq {\iota(t)} - \frac{\pi}{4k} \exp(t) \geq \frac{1}{2} \iota(t)
      \]
      for sufficiently large \(k\) 
        since \(\iota(0) = \pi\) and \(\iota\) is increasing. Consequently,
        \(\frac{\iota_k(t)}{t} \geq \frac{1}{2}\frac{\iota(t)}{t} = \frac{1}{2} f(t)\),
        and since \(|g''|\) is decreasing, it follows that
        \(|g''(\frac{\iota_k(t)}{t})|\leq |g''(\frac{1}{2}f(t))|\)
        and the right hand side is integrable on \((0,x)\) by Lemma~\ref{lemma:g-f-integrability}.
        
	Overall, it follows that \ref{Claim1} holds. As outlined before,
        this completes the proof.	
\end{proof}

In particular, this yields
Theorem~\ref{theorem:intro_second_order_expansion} which we restate here as 
 the following asymptotic expansion.
\begin{corollary}
	Let \(x>0\). Then 
	\begin{equation}\label{eq:new-expansion}
		j_{xk,k} = - \frac{4}{\pi} \frac{f(x)}{\sqrt{f(x)^2-1}} + \iota(x) k + 
		\frac{\zeta_x}{k} + O(k^{-2}) 
	\end{equation}
	as \(k \to \infty\).
\end{corollary}
We note that the existence of a related expansion for $j_{\nu, x \nu}$ has been 
suggested in
\cite{Elbert-survey}, however, a proof has not been published to our knowledge.

\begin{remark} \textbf{(Comparison to other expansions)}\\
	One may expect to recover other known expansions
	by taking \(x\) close to zero or tending to
	infinity. In order to consider the former case, first recall that
	Theorem~\ref{Theorem: Elbert-Laforgia Asymptotic Relation} and 
	Lemma~\ref{Lemma: Monotonicity and Inverse for iota quotient} yield
	\[
	\lim_{x \to 0} \iota(x) = \pi, \qquad  \lim_{x \to 0}\frac{4}{\pi} 
	\frac{f(x)}{\sqrt{f(x)^2-1}} = \frac{4}{\pi},
	\]
	respectively.
	Moreover, \eqref{eq:zeta_x_first_def} gives \(\lim_{x \to 0}\zeta_x = 
	\frac{1}{8\pi}\)
	and we thus find that as \(x \to 0\), \eqref{eq:new-expansion} essentially
	becomes McMahon's expansion,
	i.e., 
	$$
	j_{\nu,k} = k \pi + \frac{\pi}{2} \left(\nu - \frac{1}{2}\right) - 
	\frac{4\nu^2-1}{8\left(k\pi + \frac{\pi}{2} \left(\nu - 
		\frac{1}{2}\right)\right)} + O \left(\frac{1}{k^3}\right) \qquad 
		\text{as 
		$k \to \infty$}
	$$
	with \(\nu=0\).
	
	In the case \(x \to \infty\), one may expect to recover Olver's expansion
	which characterizes the asymptotic behavior of $j_{\nu,k}$ for fixed $k$ 
	as $\nu \to \infty$ and is given by (cf. \cite{Olver})
	\[
	\begin{aligned}
		j_{\nu,k}  = & \nu + \gamma_k \nu^\frac{1}{3} + \frac{3}{10} \gamma_k^2
		+ 
		\frac{3}{20} a_k^2 \frac{2^\frac{1}{3}}{\nu^\frac{1}{3}} 
		\nu^{-\frac{1}{3}} + 
		O(\nu^{-1}) \qquad \text{as $\nu \to \infty$},
	\end{aligned}
	\]
	where $\gamma_k = -{a_k}/{2^{{1}/{3}}}$ and $a_k$ is the $k$-th 
	negative zero of the Airy function.
	However, while it would be possible to analyze \eqref{eq:new-expansion} 
	for \(k=1\)
	and \(x \to \infty\), the proof of 
	Theorem~\ref{theorem:second_order_expansion} suggests
	that the error term of order $o(k^{-2})$ contains terms involving
	\( \frac{f(x)}{\sqrt{f(x)^2-1}}\), which tends to infinity as \(x \to 
	\infty\)
	by Theorem~\ref{Theorem: Elbert-Laforgia Asymptotic Relation}.
	Consequently, the relation between these expansion remains unclear.
\end{remark}

\section{Auxiliary Results}\label{sec:auxiliary-results}

\subsection{Further properties of $\xi_x$}
We discuss the sign and possible zeros of $\xi_x$ which directly affects the 
spectrum of $L_\alpha$ as outlined in Theorem~\ref{theorem:intro:main_result3}.
Recall that
\begin{equation*}
	\zeta_x = \frac{\pi}{4}\frac{f(x)}{\sqrt{f(x)^2-1}}
	\left( \frac{1}{2\pi^2}  - \int_0^x \Theta_0(t) \, dt \right)
\end{equation*} 
where the function $\Theta_0$ is given by
\begin{equation*}
	\Theta_0(x) \coloneqq -\frac{\pi}{8x^2} g''(f(x)) 
	\frac{f(x)}{\sqrt{f(x)^2-1}} -
	\frac{\del^2 F}{\del^2 s}(\iota,x,0) \frac{2}{\pi} 
	\frac{\sqrt{f(x)^2-1}}{f(x)}.
\end{equation*}
Note that computing the term $\int_0^x \Theta_0(t) \, dt$ requires information 
on the values of $f$ in $(0,x)$ which is difficult to attain due to the 
transcendental equation appearing in the definition of $\iota$ in 
\eqref{eq: iota definition identity}. To circumvent this issue, we have the 
following result.
\begin{proposition} \textbf{(Alternative formula for $\zeta_x$)} \\
  It holds that
  \[
\zeta_x = \frac{\pi}{4}\frac{f(x)}{\sqrt{f(x)^2-1}}
	\left( \frac{1}{2\pi^2}  - \int_{f(x)}^\infty \left(  - \frac{g''(t)}{8}
  + \frac{t^2 - 1}{6 \pi^2 t^4} \int_0^\infty s^3 K_1(s) e^{-\frac{s}{t} 
	}  \, 
	ds\right)   \, dt \right)
  \]
  for $x>0$.
\end{proposition}
\begin{proof}
Using the substitution $t = f^{-1}(t')$, we may write
\begin{equation} \label{eq:Theta_0-transformed-integral}
\int_0^x \Theta_0(t) \, dt = \int_{f(x)}^\infty 
\Theta_0(f^{-1}(t)) |(f^{-1})'(t)| \, dt .
\end{equation}
Then 
$$
\Theta_0(f^{-1}(t)) = - \frac{\pi}{8 (f^{-1}(t))^2} g''(t) 
\frac{t}{\sqrt{t^2-1}} - 
\frac{\del^2 F}{\del^2 s}(\iota,f^{-1}(t),0) \frac{2 }{\pi} 
\frac{\sqrt{t^2-1}}{t}.
$$
Here, we further have
$$
\begin{aligned} 
	\frac{\del^2 F}{\del^2 s}(\iota,f^{-1}(t),0) & = 
	- \frac{1}{12 \iota^2} \int_0^\infty s^3 K_1(s) e^{-\frac{x}{\iota} s}  \, 
	ds
	= - \frac{1}{12 t^2 (f^{-1}(t))^2} \int_0^\infty s^3 K_1(s) e^{-\frac{s}{t} 
	}  \, 
	ds ,
\end{aligned} 
$$
as well as
$$
g''(t) = \frac{1}{t(t^2-1)} - \frac{3\pi 
	t}{f^{-1}(t)(t^2-1)^\frac{5}{2}} 
$$ 
and 
$$
|(f^{-1})'(t)| = \frac{(f^{-1}(t))^2}{\pi} \frac{\sqrt{t^2-1}}{t} 
.
$$
Consequently, the integrand on the right hand side of 
\eqref{eq:Theta_0-transformed-integral}
is indeed given by
\[
  t \mapsto - \frac{g''(t)}{8}
  + \frac{t^2 - 1}{6 \pi^2 t^4} \int_0^\infty s^3 K_1(s) e^{-\frac{s}{t} 
	}  \, 
	ds .
      \]
as claimed.
\end{proof}
This is more advantageous in practice than the formula in 
\eqref{eq:Theta_0_def}, since the 
function $f^{-1}$ is given explicitly by \eqref{eq:f_inv-characterization} and 
in \eqref{eq:Theta_0-transformed-integral} only the value of 
$f$ in $x$ needs to be computed.

Next, we note that since \(\xi_x\) is given by the difference of two
positive numbers it would be interesting to know its sign and whether it can 
be zero. This will be particularly useful when studying the accumulation point 
in the setting of Theorem~\ref{theorem:intro:main_result3}.
To this end, we have the following characterization.
\begin{lemma}
	For $x>0$, let $\zeta_x \in \R$ be given by \eqref{eq:zeta_x_first_def}. 
	Then \(x \mapsto \zeta_x\) is continuous and satisfies
	\begin{equation} \label{eq:zeta-asymptotics}
	\lim_{x \to 0} \zeta_x = \frac{1}{8 \pi}, \qquad 
	\lim_{x \to \infty} \zeta_x = -\infty .
	\end{equation}
	Moreover, there exists a unique $x_0 \in (0,\infty)$ such that 
	$\zeta_x = 0$ if and only if $x = x_0$. 
\end{lemma}
\begin{proof}
	Recall that
	$$
	\begin{aligned} 
          \zeta_x &
	 = \frac{f(x)}{\sqrt{f(x)^2-1}} \frac{\pi}{4}
	\left( \frac{1}{2\pi^2}  - \int_0^x \Theta_0(t) \, dt  \right) .
	\end{aligned}
	$$
	Since $\Theta_0$ is strictly positive, the map $x \mapsto \int_0^x 
	\Theta_0(t) \, dt $ is strictly increasing on $(0,\infty)$ and attains the 
	value $0$ for $x=0$. 
	Moreover, $\lim_{x \to 0} f(x)=\infty$ by Lemma~\ref{Lemma: Monotonicity 
	and Inverse for iota quotient} and hence $\lim_{x \to 0} 
	\frac{f(x)}{\sqrt{f(x)^2-1}}=1$. On the other hand, Lemma~\ref{Lemma: 
	Monotonicity and Inverse for iota 
	quotient} also implies $\lim_{x \to \infty} f(x)=1$ and hence $\lim_{x \to 
	\infty} \frac{f(x)}{\sqrt{f(x)^2-1}}=\infty$. Overall, this implies
	\eqref{eq:zeta-asymptotics}.
	Regarding the existence and uniqueness of $x_0$, it suffices to show that 
	$$
	\lim_{x \to \infty} 
	 \int_0^x 
	\Theta_0(t) \, dt  >  \frac{1}{2\pi^2} .
	$$
	To this end, we note that for $t>1$,
	$$
	\Theta_0(f^{-1}(t)) \frac{|(f^{-1})'(t)|}{f^{-1}(t)}  = -\frac{1}{8} 
	g''(t) 
	+ \frac{1}{6 \pi^2} \int_0^\infty s^3 K_1(s) e^{-\frac{s}{t}} \, ds 
		\frac{t^2-1}{t^4} 
	$$
	where 
	$$
	\begin{aligned}
	- \frac{1}{8}\int_{f(x)}^\infty g''(t) \, dt = \frac{g'(f(x))}{8}
	\to \frac{1}{24} \qquad \text{as \(x \to \infty\)},
	\end{aligned}
	$$
	since $f(x) \to 1$ and $\lim_{t \downarrow 1} g'(t) = 
	\frac{1}{3}$ by Lemma~\ref{lemma:g_properties}. 
	Moreover, we may write
        $$
		\begin{aligned} 
		\int_0^\infty s^3 K_1(s) e^{-\frac{s}{t}} \, ds
		& = \int_0^\infty e^{\left(1-\frac{1}{t}\right) s} s^3 K_1(s) e^{-s} \, 
		ds \\
		& \geq \int_0^\infty \left(1+\left(1-\frac{1}{t}\right) s \right) s^3 
		K_1(s) e^{-s} \,ds \\
		& = \int_0^\infty s^3 K_1(s) e^{-s} \,ds 
		+ \left(1-\frac{1}{t}\right)\int_0^\infty s^4 K_1(s) e^{-s} \,ds
	\end{aligned}
	$$
        where for $a \in (0,1)$, Lemma~\ref{lemma:modified_bessel}(iii) implies
	$$
	\begin{aligned} 
	\int_0^\infty s^3 K_1(s) e^{-as} \, ds
	& = 	\sqrt{\frac{\pi}{2}} 
	\Gamma(3) \Gamma(5) \left(1-a^2\right)^{-\frac{7}{4}}
	P^{-\frac{7}{2}}_{\frac{1}{2}}\left( a\right) \\
	& = \sqrt{\frac{\pi}{2}} 
	\Gamma(3) \Gamma(5) (1+a)^{-\frac{7}{4}}(1-a)^{-\frac{7}{4}}
	P^{-\frac{7}{2}}_{\frac{1}{2}}\left( a\right) .
	\end{aligned}
	$$
	On the other hand, from \eqref{eq:legendre_function_def} and
        \eqref{eq:hypergeometric_function_property} it readily follows that
	$$
	\lim_{a \uparrow 1} (1-a)^{-\frac{7}{4}}
	P^{-\frac{7}{2}}_{\frac{1}{2}}\left( a\right)
        = \frac{2^{-\frac{7}{4}}}{\Gamma\left(\frac{9}{2}\right)}
        = \frac{ 2^{-\frac{7}{4}} 16}{105 \sqrt{\pi}}
	$$
	and hence
	$$
        \int_0^\infty s^3 K_1(s) e^{-s} \, ds
        =\lim_{a \uparrow 1} \int_0^\infty s^3 K_1(s) e^{-as} \, ds
        =  \frac{16}{35}  
	$$
	by Lebesgue's theorem.
        Similarly, it follows that
        $$
	\int_0^\infty s^4 K_1(s) e^{-s} \, ds =  \frac{16}{21}  .
	$$
        Consequently,
        $$
        \int_1^\infty \frac{t^2-1}{t^4}  \int_0^\infty s^3 K_1(s) e^{-\frac{s}{t}} \, ds \, dt
        \geq \frac{2}{3} \frac{16}{35} + \frac{5}{12} \frac{16}{21} = \frac{28}{45} 
        $$
        and therefore
        $$
        \lim_{x \to \infty}  \int_0^x \Theta_0(t) \, dt
        \geq \frac{1}{24} + \frac{14}{135 \pi^2} > \frac{1}{2\pi^2},
        $$
        as claimed.
\end{proof}

\begin{remark}
	\begin{itemize}
		\item[(i)]
                  The previous result shows that $\zeta_x = 0$ when
                $$
		\int_0^x \Theta_0(t) \, dt = \frac{1}{2\pi^2}.
		$$
                Computations suggest that this is the case for
                $f(x) \approx 1.384...$, which 
		corresponds to $x \approx 16.237924160981667...$.
                The explicit value could possibly be computed by using the 
                series expansion of
                the exponential function to write
                \[
                  \begin{aligned}
                  &\int_{f(x)}^\infty \frac{t^2-1}{t^4}
                  \int_0^\infty s^3 K_1(s) e^{-\frac{s}{t}} \, ds \, dt \\
                 & = \sum_{k=0}^\infty \frac{1}{k!} \int_0^\infty s^{3+k} K_1(s) e^{-s} \, ds
                  \int_{f(x)}^\infty \frac{t^2-1}{t^4} \left(1-\frac{1}{t^2}\right)^k \, dt 
                  \end{aligned}
                \]
                and hence \(\zeta_x = 0\) if and only if
                \[
                  \begin{aligned}
                    \frac{1}{2\pi^2}
                    & = \frac{g'(f(x))}{8}
                      + \frac{1}{6 \pi^2}  \sum_{k=0}^\infty \frac{1}{k!}
                      \int_0^\infty s^{3+k} K_1(s) e^{-s} \, ds
                  \int_{f(x)}^\infty \frac{(t^2-1)^{k+1}}{t^{2k+4}} \, dt .
                   \end{aligned}
                 \]
                 
		\item[(ii)]
                  Computations suggest that the third order term in the expansion
                  \eqref{eq:second_order_theorem_statement} is positive.
                  In fact, the previous strategy can be used to successively compute
                  the next coefficients.
	\end{itemize}
\end{remark}

\subsection{Asymptotic behavior of derivatives}
While we already established a useful characterization of the behavior of 
$j_{\sigma k,k}$ as $k \to \infty$, we will later also need similar information 
on the asymptotics of its derivative. More specifically, we consider the 
derivative of the function $(\nu, k) \mapsto j_{\nu,k}$ with respect to $\nu$ 
and give the following result.
\begin{lemma} \label{lemma:derivative_convergence_order}
	Let $\sigma>0$ and let $(\delta_k)_k$ be a sequence of nonnegative real numbers such 
	that $\frac{\delta_k}{k} \to 0$ as $k \to \infty$. Then there exist 
	constants 
	$k \in N$, $C_\sigma>0$ such that 
	$$
	\left| 
	\frac{\partial j_{\nu,k}}{\partial \nu} (\sigma k -\delta_k, k) - 
	g(f(\sigma)) -\left(-\frac{\pi }{4\sigma} 
	\frac{f(\sigma)}{\sqrt{f(\sigma)^2-1}} 
	g'(f(\sigma)) - g'(f(\sigma)) f'(\sigma) {\delta_k}\right) 
	\frac{1}{k} \right| \leq  \frac{C_\sigma}{k^2} 
	$$
	for $k \geq k_0$, where $g$ is given by 
	$$
	g:(1,\infty) \mapsto \R, \quad t \mapsto \frac{\arccos 
		\frac{1}{t}}{\sqrt{1-\frac{1}{t^2}}} .
	$$
\end{lemma}
\begin{proof}
%	\textbf{Claim 1:} .... 
%	
	We set
	$$
	F(x, s)  \coloneqq 
	\begin{cases} 
		\int_0^\infty K_0 (t) 
		\frac{\exp\left(-\frac{x}{s}\arsinh(ts)\right)}{\sqrt{1+t^2s^2}} \, dt 
		\qquad 
		&\text{for $x \neq 0$}, \\
		\int_0^\infty K_0 (t) 
		\exp(-xt) \, dt
		& \text{for $x=0$.}
	\end{cases}
	$$
        Recall that
        $$
        \frac{\partial j_{\nu,k}}{\partial \nu} (\sigma k -\delta_k, k)
        = 2 j_{\sigma k -\delta_k,k}
        \int_0^\infty K_0 (2  j_{\sigma k -\delta_k,k}  \sinh(t)) e^{-2 \nu t} \, dt ,
        $$
        so that substituting \(u = 2 j_{\sigma k -\delta_k,k} \sinh t \) gives
	$$
	\frac{\del}{\del \nu} j_{\sigma k -\delta_k,k} = \int_0^\infty K_0(u) 
	\frac{\exp\left(-2 \nu 
            \arsinh\left(\frac{u}{2j_{\sigma k -\delta_k,k}}\right)\right)}
        {\sqrt{1+\frac{u^2}{4j_{\sigma k -\delta_k,k}^2}}} \, du
        = F\left(\frac{\sigma k -\delta_k}{j_{\sigma k 
		-\delta_k,k}}, \frac{1}{2j_{\sigma k 
		-\delta_k,k}}\right).
	$$
	From Lemma~\ref{lemma:integral_taylor_expansion_2nd}, it follows that
	$$
	\begin{aligned}
		F\left(\frac{\sigma k -\delta_k}{j_{\sigma k 
				-\delta_k,k}}, \frac{1}{2j_{\sigma k 
				-\delta_k,k}}\right)
		& = F\left(\frac{\sigma k -\delta_k}{j_{\sigma k 
				-\delta_k,k}}, 0 \right) + R_k(\sigma k -\delta_k)
	\end{aligned}
	$$
	with 
	$$
	R_k(\sigma k -\delta_k) = \int_0^\frac{1}{2j_{\sigma k 
			-\delta_k,k}} \frac{\del^2 F}{\del^2 s}\left(\frac{\sigma k 
			-\delta_k}{j_{\sigma k 
			-\delta_k,k}}, t\right) 
			\left(\frac{1}{2j_{\sigma k 
			-\delta_k,k}}-t\right) \, dt.
	$$
	By Lemma~\ref{lemma:old_iota_bounds-uniform}, there exist constants $C_\sigma>1$,  
	$k_0 \in \N$ such that 
	$\frac{\sigma k -\delta_k}{j_{\sigma k 
			-\delta_k,k}} \in (C_\sigma^{-1}, C_\sigma)$ for $k \geq k_0$, and 
			after further enlarging $k_0$, we may assume $\sigma k -\delta_k > 
			\frac{\sigma k}{5}$ for $k \geq k_0$.
	Additionally, the continuity of $\frac{\del^2 F}{\del^2 s}$ implies that 
	there exists 
	$$
	M_0 \coloneqq \max \left\{ \left|\frac{\del^2 F}{\del^2 s}(s,t)\right|: 
	(s,t) \in \left[C_\sigma^{-1}, C_\sigma\right] \times [0,1]\right\} < 
	\infty 
	$$
	and using the bound
        $j_{\sigma k -\delta_k,k} > \sigma k -\delta_k + \pi k - 1$
        due to Ifantis and Siafarikas \cite{Ifantis-Siafarikas}, this yields
	$$
	|	R_k(\sigma k -\delta_k)| \leq M_0
	\frac{1}{8 j_{\sigma k -\delta_k,k}^2} \leq 
	\frac{1}{8\left(\frac{\sigma}{5} 
	+ \pi - 1\right)^2 k^2} 
	$$
	for $k \geq k_0$.
	Moreover, 
	\begin{equation} \label{eq:F_g_identity} 
	F\left(\frac{\sigma k -\delta_k}{j_{\sigma k 
			-\delta_k,k}}, 0 \right) = g\left(\frac{j_{\sigma k 
			-\delta_k,k}}{\sigma k -\delta_k}\right)
	\end{equation}
	by Lemma~\ref{lemma:g_properties}.
	We may then write
	\begin{equation} \label{eq:g_circ_f_target_differences}
          \begin{aligned}
           &g\left(\frac{j_{\sigma k 
			-\delta_k,k}}{\sigma k -\delta_k}\right) \\
	& = g(f(\sigma)) + \left(g\left(\frac{j_{\sigma k 
			-\delta_k,k}}{\sigma k -\delta_k}\right) - 
			g\left(f\left(\sigma-\frac{\delta_k}{k}\right)\right)\right) 
			+ \left(g\left(f\left(\sigma-\frac{\delta_k}{k}\right)\right)  
            - g(f(\sigma)) \right)
          \end{aligned}
	\end{equation}
	and further analyze the two differences appearing on the right hand side.
	Firstly, we note that
	$$
	\begin{aligned} 
	& g\left(\frac{j_{\sigma k 
			-\delta_k,k}}{\sigma k -\delta_k}\right) - 
	g\left(f\left(\sigma-\frac{\delta_k}{k}\right)\right)  \\
	& = 
	g'\left(f\left(\sigma-\frac{\delta_k}{k}\right)\right)
	\left(\frac{j_{\sigma k 
			-\delta_k,k}}{\sigma k -\delta_k} - 
	f\left(\sigma-\frac{\delta_k}{k}\right)\right)
	+
	\int_{f\left(\sigma-\frac{\delta_k}{k}\right)}^{\frac{j_{\sigma k 
				-\delta_k,k}}{\sigma k -\delta_k}} g''(t)\left({\frac{j_{\sigma 
				k 
				-\delta_k,k}}{\sigma k -\delta_k}}-t\right) \, dt
	\end{aligned}
	$$
	where Lemma~\ref{lemma:g_properties} and the fact that the maps $f$ and 
	$\nu \mapsto \frac{j_{\nu,k}}{\nu}$ are decreasing imply
	$$
	\begin{aligned}
	\left|\int_{f\left(\sigma-\frac{\delta_k}{k}\right)}^{\frac{j_{\sigma k 
				-\delta_k,k}}{\sigma k -\delta_k}} g''(t)\left({\frac{j_{\sigma 
				k 
				-\delta_k,k}}{\sigma k -\delta_k}}-t\right) \, dt\right|
			& \leq \frac{2}{5} 
			\int_{f\left(\sigma-\frac{\delta_k}{k}\right)}^{\frac{j_{\sigma k 
						-\delta_k,k}}{\sigma k 
						-\delta_k}}\left({\frac{j_{\sigma k 
						-\delta_k,k}}{\sigma k -\delta_k}}-t\right) \, dt	\\
					& = \frac{1}{5} \left|  \frac{j_{\sigma k 
					-\delta_k,k}}{\sigma k -\delta_k} - 
					f\left(\sigma-\frac{\delta_k}{k}\right) \right| \\
					& \leq \frac{1}{5 \sigma  } \left|  \frac{j_{\sigma k 
					-\delta_k,k}}{k} - 
					\iota\left(\sigma-\frac{\delta_k}{k}\right) 
					\right| 
				\end{aligned}
	$$
	for $k \geq k_0$.
	After possibly enlarging $k_0$, Lemma~\ref{lemma:old_iota_bounds-uniform} then yields
	\begin{equation} \label{eq:g''_int_estimate}
	\left|\int_{f\left(\sigma-\frac{\delta_k}{k}\right)}^{\frac{j_{\sigma k 
				-\delta_k,k}}{\sigma k -\delta_k}} g''(t)\left({\frac{j_{\sigma 
				k 
				-\delta_k,k}}{\sigma k -\delta_k}}-t\right) \, dt\right|	
	\leq \frac{\pi \exp \left(  \sigma \right)}{10 \sigma k^2}  
	\end{equation}
	for $k \geq k_0$.
	Next, we note that
	$$
	g'\left(f\left(\sigma-\frac{\delta_k}{k}\right)\right) = g'(f(\sigma)) + 
	g''(f(\xi)) f'(\xi) \frac{\delta_k}{k}
	$$
	for some $\xi \in \left(\sigma - \frac{\delta_k}{k}, \sigma\right)$, where 
	monotonicity yields
	$$
	\left|	g''(f(\xi)) f'(\xi) \right| \leq \frac{2}{5} 
	\left|f'\left(\frac{\sigma}{5}\right)\right|
	$$
	for $k \geq k_0$. Similarly, it follows that 
	$$
		\left|\frac{1}{\sigma - \frac{\delta_k}{k}} \right| \leq 
		\frac{5}{\sigma} 
		\frac{\delta_k}{k}	
	$$
	for $k \geq k_0$, while Theorem~\ref{theorem:second_order_expansion} gives
	$$
	\left|\frac{j_{\sigma k -\delta_k,k}}{k} - 
	\iota\left(\sigma-\frac{\delta_k}{k} 
	\right) 
	+\frac{\pi}{4k\left(\sigma-\frac{\delta_k}{k} \right) 
	}\frac{f\left(\sigma-\frac{\delta_k}{k} \right) 
	}{\sqrt{f\left(\sigma-\frac{\delta_k}{k} \right) ^2-1}} \right| \leq 
	\frac{C_\sigma}{k^2}
	$$
	for $k \geq k_0$, 
	after possible enlarging $C_\sigma$.
	Noting that
	$$
	\begin{aligned}
	& g'\left(f\left(\sigma-\frac{\delta_k}{k}\right)\right)
	\left(\frac{j_{\sigma k 
			-\delta_k,k}}{\sigma k -\delta_k} - 
	f\left(\sigma-\frac{\delta_k}{k}\right)\right) \\
	& = g'\left(f\left(\sigma-\frac{\delta_k}{k}\right)\right)
	\frac{1}{\sigma - \frac{\delta_k}{k}}
	\left( \frac{j_{\sigma k 
			-\delta_k,k}}{k} - \iota\left(\sigma-\frac{\delta_k}{k}\right) 
			\right) ,
	\end{aligned} 
	$$
	we thus find that
		$$
	\begin{aligned}
		& \left|  g'\left(f\left(\sigma-\frac{\delta_k}{k}\right)\right)
		\left(\frac{j_{\sigma k 
				-\delta_k,k}}{\sigma k -\delta_k} - 
		f\left(\sigma-\frac{\delta_k}{k}\right)\right) 
		+\frac{\pi}{4k\left(\sigma-\frac{\delta_k}{k} \right) 
		}\frac{f\left(\sigma-\frac{\delta_k}{k} \right) 
		}{\sqrt{f\left(\sigma-\frac{\delta_k}{k} \right) ^2-1}}
		\right|
		\leq \frac{C_\sigma}{k^2} ,
	\end{aligned} 
	$$
	and analogous arguments then yield
	\begin{equation} \label{eq:g'_circ_f_big_estimate}
	\begin{aligned}
		& \left|  g'\left(f\left(\sigma-\frac{\delta_k}{k}\right)\right)
		\left(\frac{j_{\sigma k 
				-\delta_k,k}}{\sigma k -\delta_k} - 
		f\left(\sigma-\frac{\delta_k}{k}\right)\right) 
		+\frac{\pi}{4k\sigma 
		}\frac{f\left(\sigma \right) 
		}{\sqrt{f\left(\sigma \right) ^2-1}}
		\right|
		\leq \frac{C_\sigma}{k^2} .
	\end{aligned} 
	\end{equation}
	Similarly, it follows that there exists $\xi \in \left(\sigma - 
	\frac{\delta_k}{k}, \sigma\right)$ such that 
      \begin{equation} \label{eq:g_circ_f_estimate}
        \begin{aligned}
	& g\left(f\left(\sigma-\frac{\delta_k}{k}\right)\right)  
	- g(f(\sigma)) \\
	& = - g'(f(\sigma)) f'(\sigma) \frac{\delta_k}{k}
          + (-g''(\xi) (f'(\xi))^2 - g'(f(\xi)) f''(\xi)) \frac{\delta_k^2}{k^2}
        \end{aligned}
	\end{equation}
	where
	$$
	\left|-g''(\xi) (f'(\xi))^2 - g'(f(\xi)) f''(\xi)\right| \leq \frac{2 	
	\left|f'\left(\frac{\sigma}{5}\right)\right| }{5} +  \frac{ 	
	\left|f''\left(\frac{\sigma}{5}\right)\right| }{3}
	$$
	Applying the estimates \eqref{eq:g''_int_estimate}, 
	\eqref{eq:g'_circ_f_big_estimate} and \eqref{eq:g_circ_f_estimate} in 
	\eqref{eq:g_circ_f_target_differences}, it follows that 
	$$
	\begin{aligned} 
	\left| g\left(\frac{j_{\sigma k 
		-\delta_k,k}}{\sigma k -\delta_k}\right) 
	-  g(f(\sigma)) - \left(-\frac{\pi }{4\sigma} 
	\frac{f(\sigma)}{\sqrt{f(\sigma)^2-1}} 
	g'(f(\sigma)) - g'(f(\sigma)) f'(\sigma) {\delta_k}\right) 
	\frac{1}{k} \right| \leq  \frac{C_\sigma}{k^2} 
	\end{aligned} 
	$$
	hold for $k \geq k_0$ and, recalling \eqref{eq:F_g_identity}, this 
	completes the proof.
\end{proof}
\begin{remark}\label{remark:delta_special_case}
In the special case \(\delta_k = \frac{\sigma}{4}\), we have
$$
-\frac{\pi }{4\sigma} 
\frac{f(\sigma)}{\sqrt{f(\sigma)^2-1}} 
-  f'(\sigma) \frac{{\sigma}}{4} = -\frac{\pi }{4f^{-1}(\alpha)} 
\frac{\alpha}{\sqrt{\alpha^2-1}} 
-  \frac{1}{(f^{-1})'(\alpha)} \frac{{f^{-1}(\alpha)}}{4} = 0
$$
since
$$
\begin{aligned}
	\frac{(f^{-1})'(\alpha)}{f^{-1}(\alpha)}   
	=-\frac{\frac{\sqrt{\alpha^2-1}}{\alpha}}{\sqrt{\alpha^2- 1} - 
		\arccos \frac{1}{\alpha} } =- \frac{\sqrt{\alpha^2-1}}{\alpha} 
		\frac{f^{-1}(\alpha)}{\pi} .
\end{aligned}
$$
Consequently, in this case the statement of Lemma~\ref{lemma:derivative_convergence_order}
simplifies to
	$$
	\left| 
	\frac{\partial j_{\nu,k}}{\partial \nu} (\sigma k -\delta_k, k) - 
	g(f(\sigma))  \right| \leq  \frac{C_\sigma}{k^2} 
	$$
	for $k \geq k_0$.
\end{remark}
%
%%%%%%%%%%%%%%%%%%%%%%%%%%%%%%%%%%%%%%%%%%%%%%%%%%%%%%%%%%%%%%%%%%%%%%%%%%%%%%%%%%%%%%%%%%%%%%%%%%%%%%%%%%%%%
%
%
%
%
\section{Results for Rotating Waves}\label{Section:Spectral-Characterization}
Subsequently, we use the new estimates for Bessel function zeros to study the
spectral and variational properties of \eqref{Reduced equation}.
\subsection{Spectrum of $L_\alpha$} 
Recall the operator 
\[
L_\alpha=-\Delta + \alpha^2 \del_{\theta}^2
\]
for $\alpha>1$
where \(\del_\theta= x_1 \del_2 - x_2 \del_1 \) corresponds to the angular 
derivative. As outlined in Section~\ref{Section: Preliminaries}, 
the Dirichlet eigenvalues of $L_\alpha$ are given by 
\[
\left\{ j_{\ell,k}^2-\alpha^2 \ell^2 : \ \ell \in \N_0, \, k \in \N \right\}.
\]
It was established in \cite[Proposition 4.1]{Kuebler}, that
the spectrum of $L_\alpha$ is unbounded from above and below for any \(\alpha>1\). 
In general, it is unclear whether the spectrum of $L_\alpha$
contains accumulation points.

In \cite[Theorem 4.2]{Kuebler}, it was established that the spectrum contains
no accumulation points, provided \(f^{-1}(\alpha)=\frac{1}{n}\) for some sufficiently large 
\(n \in \N\), where \(f^{-1}\) is given by \eqref{eq:f_inv-characterization}.
Additionally, \(j_{\ell,k}^2 - \alpha^2 \ell^2\) grows at least as fast as ${j_{\ell,k}}$ in this
case.
In the following, we improve this result, starting with the proof of 
Theorem~\ref{theorem:intro:main_result} which we slightly restate here as 
follows.
\begin{theorem}  \label{Theorem: Spectrum has no accumulation point}
	 Let $\alpha>1$ such that $\sigma := f^{-1}(\alpha)>0$ is rational and let 
	 $q_1, q_2 \in \N$ coprime such that $\sigma = \frac{q_1}{q_2}$. Additionally, 
	 assume that at least one of the following two conditions holds:
	 \begin{itemize}
	 \item[\namedlabel{condition:C1-end}{\textit{(C1)}}]
	 4 is not a divisor of $q_1$.
	 \item[\namedlabel{condition:C2-end}{\textit{(C2)}}]
	 $q_2$ is even.
	\end{itemize}
Then, the following statements hold:
	 \begin{itemize}
	 	\item[(i)] The spectrum of $L_{\alpha}$ consists of eigenvalues with 
	 	finite multiplicity.
	 	\item[(ii)] There exists $c>0$ 
	 	such that for each 
	 	$\ell \in \N_0$, $k \in \N$ we either have $j_{\ell,k}^2-\alpha^2 
	 	\ell^2=0$ or
	 	\begin{equation} \label{eq: eigenvalue lower bound-restated} 
	 	|j_{\ell,k}^2-\alpha^2 \ell^2| \geq c j_{\ell,k}.
              \end{equation}
	 	\item[(iii)] The spectrum of $L_{\alpha}$ has no finite 
	 	accumulation points.
	 \end{itemize}
\end{theorem}
\begin{proof}
	In the following, we assume that there 
	exists $\Lambda \in \R$ and increasing sequences $(\ell_i)_i$, $(k_i)_i$ 
	such that $j_{\ell_i,k_i}^2-\alpha^2 \ell_i^2 \to \Lambda$ as $n \to 
	\infty$. 
	Completely analogous to \cite[Proof of Theorem 4.2]{Kuebler} it can be shown that this is only 
	possible if $\frac{\ell_i}{k_i} \to \sigma$ and $j_{\ell_i,k_i}-\alpha 
	\ell_i \to 0$ as $i \to \infty$. 
	
	If $\frac{\ell_i}{k_i} \geq \sigma$ for large $i$,
        Lemma~\ref{lemma:new_iota_bounds} yields
	\[
	\left(\frac{j_{\ell_i,k_i}}{\ell_i}-\alpha  \right)
        \leq \frac{1}{\sigma} \left( j_{\sigma k_i,k_i}- \iota(\sigma) \right)
        \leq  - \frac{\pi }{4\sigma k_i} \frac{\alpha}{\sqrt{\alpha^2-1}}
        + \frac{C_\sigma}{\sigma k_i^2}
     \]
        and consequently,
    \[
	\liminf_{i \to  \infty} |j_{\ell_i,k_i} - \alpha \ell_i| 
	= \frac{\pi}{4} \frac{\alpha}{\sqrt{\alpha^2-1}} >0
	\]
	for such sequences.
	
	It remains to consider the case where
	$\ell_i = \sigma k_i - \delta_i$ with $\delta>0$ satisfying $\frac{\delta_i}{k_i} \to 0$ 
	as $i \to \infty$.
	Then
    \begin{align*}
    	j_{\ell_i,k_i}-\alpha \ell_i & =
    	j_{(\sigma k_i - \delta_i),k_i} - \alpha (\sigma k_i-\delta_i)  \\
    	&= (j_{\sigma k_i,k_i} -\alpha \sigma k_i) + (j_{(\sigma k_i - 
    	\delta_i),k_i} - j_{\sigma k_i ,k_i}) + \alpha \delta_i \\
    	& =    (j_{\sigma k_i,k_i} -\alpha \sigma k_i)  + R_{i} 
    	\delta_i ,	
    \end{align*}
    where we have set
    \[
    R_{i} \coloneqq \alpha - \frac{j_{\sigma k_i ,k_i}-j_{\sigma k_i - 
    \delta_i ,k_i}}{\delta_i} .
    \]
    Similar to the previous argument, Lemma~\ref{lemma:new_iota_bounds} gives 
    \[
    j_{\sigma k_i,k_i} -\alpha \sigma k_i = -\frac{\pi}{4} \frac{\alpha}{\sqrt{\alpha^2-1}} + o(1)
    \]
    as $i \to \infty$, while
    \[
    \frac{j_{\sigma k_i ,k_i}-j_{\sigma k_i -  \delta_i ,k_i}}{\delta_i}
    = \frac{\partial j_{\nu,k}}{\partial \nu} (\xi, k)
    \]
    for some \(\xi \in (\sigma k -\delta_k,\sigma k)\) and hence 
    Lemma~\ref{lemma:derivative_convergence_order} gives
    \[
    \frac{ j_{\sigma k_i ,k_i} - j_{\sigma k_i-\delta_i ,k_i}}{\delta_i} 
    = \frac{\alpha\arccos \frac{1}{\alpha}}{\sqrt{\alpha^2-1}} + o(1)
    \]
    as $i \to \infty$.
    Note that
    \[
    \alpha - \frac{\alpha\arccos \frac{1}{\alpha}}{\sqrt{\alpha^2-1}} 
    = \frac{\alpha}{\sqrt{\alpha^2-1}} \left( \sqrt{\alpha^2-1} - \arccos \frac{1}{\alpha} \right) 
    = \frac{\pi \alpha}{\sigma \sqrt{\alpha^2-1}} .
    \]
    Overall, we thus find that
     \begin{align*} 
    	j_{\ell_i,k_i}-\alpha \ell_i & =  -\frac{\pi}{4 } 
    	\frac{\alpha}{\sqrt{\alpha^2-1}} + \delta_i \frac{\pi \alpha }{\sigma 
    	\sqrt{\alpha^2-1}}  + o(1).
    \end{align*}
	Hence, $j_{\ell_i,k_i}-\alpha \ell_i \to 0$ is only possible if
	\begin{equation} \label{delta_convergence_condition}
	\delta_i \to \frac{\sigma}{4} \quad \text{	as $i \to \infty$.}
	\end{equation}
	Since $\sigma$ is rational by assumption, 
	$\delta_i = \sigma k_i - \ell_i = \frac{q_1 k_i - q_2 \ell_i}{q_2}$ only attains 
	finitely many distinct values in 
	$\left(\frac{\sigma}{4}-\frac{1}{2}, \frac{\sigma}{4}+\frac{1}{2}\right)$.
	Consequently, \eqref{delta_convergence_condition} is only possible if 
	$\delta_i = \frac{\sigma}{4}$ for $i$ large enough.
	Additionally, since $\sigma k_i - \delta_i = \ell_i$ must be an integer and we may write 
	\[
	\sigma k_i - \delta_i  = \frac{q_1 k_i}{q_2} - \frac{q_1}{4q_2} = \frac{q_1 (4k_i-1)}{4q_2} ,
	\]
	it follows that this is only possible if and only if 
	$q_1$ is divisible by 4 and $q_2$ is odd.
	Indeed, it is impossible for $q_2$ to divide $q_1$ or for $4$ to divide $4k_i-1$. 
	On the other hand, if $q_2$ is odd, $q_2$ can be written as $4k_0 -1$ or $4k_0+1$ for some $k_0 \in \N$. If $q_2=4k_0-1$ and $i \in \N$ is arbitrary, it follows that
	\[
	(4k_0-1)(4i+1) = (16k_0 i - 4i+4k_0-1) = 4k_i - 1  
	\]	
	for $k_i=4k_0 i +k_0-i \in \N$ and thus $q_2$ divides $4k_i - 1$.
	If $q_2=4k_0+1$, it follows similarly that $q_2$ divides $4k_i' - 1$
        for $k_i'=4k_0 i -k_0+i \in \N$.
        
	Since these cases are not possible if \ref{condition:C1-end} or 
	\ref{condition:C2-end}
	hold, it follows that 
	\[
	\liminf_{i \to  \infty}|	j_{\ell_i,k_i}-\alpha \ell_i| > 0 
	\]    
    for any sequences $(\ell_i)_i,(k_i)_i$ such that $\frac{\ell_i}{k_i} 
    \to \sigma $. Consequently, $j_{\ell,k}^2-\alpha^2 \ell^2 = (j_{\ell,k}-\alpha \ell)(
    j_{\ell,k}+\alpha \ell)$ cannot converge to $\Lambda$,
    implying (i) and (iii).

    The remainder of the proof is completely analogous to
    \cite[Proof of Theorem 4.2]{Kuebler}.
\end{proof}
Next, we restate and prove Theorem~\ref{theorem:intro:main_result3}.
\begin{theorem}  
	Let $\alpha>1$ such that $\sigma := f^{-1}(\alpha)>0$ is rational and let 
	$q_1,q_2 \in \N$ coprime such that $\sigma = \frac{q_1}{q_2}$. Additionally, 
	assume that 
	\begin{itemize}
		\item[\namedlabel{condition:C3-end}{\textit{(C3)}}]
		4 is a divisor of $q_1$ and $q_2$ is odd.
              \end{itemize}
              Then
              \[
              \Sigma_* := \left\{j_{\ell,k}^2-\alpha^2 \ell^2:
                k \in \N, 
                \ell = \sigma k - \frac{\sigma}{4} \in \N  \right\}
              \]
              is nonempty and \(\Sigma_* \subset \Sigma\).
              In particular, $\Sigma_*$ consists of a single sequence and this 
              sequence converges to \(2 \alpha \sigma \zeta_\sigma\).
              
              Moreover, the remainder of the spectrum has the following properties:
	\begin{itemize}
		\item[(i)] $\Sigma \setminus \Sigma_*$ is unbounded from above and below and 
		exclusively consists of eigenvalues with 
		finite multiplicity.
		\item[(ii)] There exists $c>0$ 
		such that for each 
		$\ell \in \N_0$, $k \in \N$ such that $j_{\ell,k}^2-\alpha^2 \ell^2 \in 
		\Sigma \setminus \Sigma_*$, we either have $j_{\ell,k}^2-\alpha^2 
		\ell^2=0$ or
		\begin{equation}
			|j_{\ell,k}^2-\alpha^2 \ell^2| \geq c j_{\ell,k}.
		\end{equation} 
		\item[(iii)] $\Sigma \setminus \Sigma_*$ has no finite 
		accumulation points.
	\end{itemize}
      \end{theorem}
\begin{proof}
  By assumption \ref{condition:C3-end}, \(\Sigma_2\) is nonempty. Writing
  $\ell_i = \sigma k_i - \delta_i$ with \(\delta_i = \frac{\sigma}{4}\),
  we may use Remark~\ref{remark:delta_special_case} to deduce 
  \[
    \frac{ j_{\sigma k_i ,k_i} - j_{\sigma k_i-\delta_i ,k_i}}{\delta_i} 
    = \frac{\alpha\arccos \frac{1}{\alpha}}{\sqrt{\alpha^2-1}} + O(k_i^{-2}),
    \]
    and Theorem~\ref{theorem:second_order_expansion} yields
  \[
  j_{\sigma k_i,k_i} -\alpha \sigma k_i = -\frac{\pi}{4} \frac{\alpha}{\sqrt{\alpha^2-1}}
  + \frac{\zeta_\sigma}{k_i} + o(k_i^{-1})
  \]
    as $i \to \infty$.
     Overall, we thus find that
     \begin{align*} 
    	j_{\ell_i,k_i}-\alpha \ell_i & = \frac{\zeta_\sigma}{k_i} + o(k_i^{-1}) 
    \end{align*}
    and thus
    \[
    \left(j_{\ell_i,k_i}^2 -\alpha^2 \ell_i^2 \right)
    = \zeta_\sigma \left( \frac{j_{\ell_i,k} + \alpha \ell_i}{k_i}\right) + o(1)
    = 2 \alpha \sigma \zeta_\sigma + o(1)
    \]
    as \(i \to \infty\).
    The remaining assertions follow analogously to the previous theorem.
\end{proof}

\subsection{Symmetry breaking for ground states}
We briefly recall sketch how ground states for \eqref{Reduced equation} can be 
defined. We refer to \cite{Kuebler} for details.

For $\alpha>1,m \in \R$, we set 
\[
L_{\alpha,m} \coloneqq -\Delta + \alpha^2 \del_\theta + m .
\]
\begin{remark}\label{remark:self_adjoint_question}
  Using integration by parts, it is easy to see that the operator
  \(L_{\alpha,m}\) is symmetric on \(C_c^\infty\) but due to the
  lack of semiboundedness, it is unclear whether the operator is
  essentially self-adjoint.
  We can, however, consider a specific self-adjoint extension on
  \(H^2(\B) \cap H^1_0(\B)\) analogous to the Dirichlet Laplacian.
  Note that the two operators also share an orthonormal basis
  of eigenfunctions, see \(\phi_{\ell,k}, \psi_{\ell,k}\) below. 
  In the following we view the operator \(L_{\alpha,m}\) always
  as this specific self-adjoint extension. 
  
\end{remark}
Recall that in polar coordinates and with suitable constants $A_{\ell,k}, 
B_{\ell,k}>0$, the functions 
\[
\begin{aligned}
	\phi_{\ell,k}(r,\theta ) & = A_{\ell,k} \cos(\ell \theta) 
	J_{\ell}(j_{\ell,k}r) \\	
	\psi_{\ell,k}(r,\theta ) & = B_{\ell,k} \sin(\ell \theta) 
	J_{\ell}(j_{\ell,k}r) 
\end{aligned} 
\]
constitute an orthonormal basis of $L^2(\B)$ for $\ell \in \N_0, k \in \N$.
Moreover, $\phi_{\ell,k}, \psi_{\ell,k}$ are eigenfunctions of $L_{\alpha,m}$
corresponding to the eigenvalue $j_{\ell,k}^2 - \alpha^2 \ell^2 + m$.
We therefore define 
\[
E_{\alpha,m} \coloneqq \left\{
u \in L^2(\B): \sum_{\ell=0}^\infty \sum_{k=0}^\infty |j_{\ell,k}^2 - \alpha^2 
\ell^2 + m| (|\langle u, \phi_{\ell,k} \rangle|^2+ |\langle u, \psi_{\ell,k} 
\rangle|^2 ) < \infty
\right\}
\]
and endow $E_{\alpha,m}$ with the scalar product
\[
\begin{aligned}
\langle u, v \rangle_{\alpha,m} \coloneqq &
 \sum_{\ell=0}^\infty \sum_{k=0}^\infty |j_{\ell,k}^2 - \alpha^2 
\ell^2 + m| \left(\langle u, \phi_{\ell,k} \rangle \langle v, \phi_{\ell,k} 
\rangle 
+ \langle u, \psi_{\ell,k} \rangle \langle v, \psi_{\ell,k} \rangle \right) \\
& + \sum_{\substack{(\ell,k) \in \N_0 \times \N: \\ j_{\ell,k}^2 - \alpha^2 
\ell^2 + m = 0}}
\left(\langle u, \phi_{\ell,k} \rangle \langle v, \phi_{\ell,k} \rangle 
+ \langle u, \psi_{\ell,k} \rangle \langle v, \psi_{\ell,k} \rangle \right) .
\end{aligned}
\]
Here \(\langle \cdot, \cdot \rangle\) denotes the
scalar product on \(L^2(\B)\).
\begin{lemma}
  	Let $p \in (2,4)$, $m \in \R$ and assume $\alpha>1$ satisfies the 
	conditions of Theorem~\ref{Theorem: Spectrum has no accumulation point}.
        Then \(E_{\alpha, m}\) is a real Hilbert space.
\end{lemma}
\begin{proof}
  The properties of the scalar product on \(L^2\) readily imply that \(E_{\alpha,m}\)
  is a real vector space and that  \(\langle \cdot, \cdot \rangle_{\alpha,m}\)
  defines a scalar product on \(E_{\alpha,m}\). It therefore only remains to show
  that the space is complete.
  To this end, consider a Cauchy sequence \((u_n)_n\) in \(E_{\alpha,m}\) and
  note that \eqref{eq: eigenvalue lower bound-restated} implies that \((u_n)_n\)
  is also a Cauchy sequence with respect to the \(L^2\)-norm and hence converges
  in \(L^2\) to a function \(u_0 \in L^2(\B)\).
  For any \(M \in \N\) it follows that
  \[
    \begin{aligned}
& \sum_{\ell=0}^M \sum_{k=0}^M |j_{\ell,k}^2 - \alpha^2 
\ell^2 + m| (|\langle u_0, \phi_{\ell,k} \rangle|^2+ |\langle u_0, \psi_{\ell,k} 
\rangle|^2 ) \\
= & \lim_{n \to \infty} \sum_{\ell=0}^M \sum_{k=0}^M |j_{\ell,k}^2 - \alpha^2 
\ell^2 + m| (|\langle u_n, \phi_{\ell,k} \rangle|^2+ |\langle u_n, \psi_{\ell,k} 
\rangle|^2 ) \\
\leq & \sup_n \|u_n\|_{\alpha,m}^2 < \infty
    \end{aligned}
  \]
  and hence \(u_0 \in E_{\alpha,m}\) by monotone convergence.
  With a similar argument, \(u_n \to u_0\) in \(E_{\alpha,m}\)
  as \(n \to \infty\) can be shown. This completes the proof.  
\end{proof}

\begin{remark}
  The space \(E_{\alpha,m}\) corresponds to the form domain of the specific
  self-adjoint extension of \(L_{\alpha,m}\) discussed in
  Remark~\ref{remark:self_adjoint_question}.
\end{remark}

Importantly, the estimate \eqref{eq: eigenvalue lower bound-restated}  yields
an embedding for the space $E_{\alpha,m}$.
\begin{proposition}
	Let $p \in (2,4)$, $m \in \R$ and assume $\alpha>1$ satisfies the 
	conditions of Theorem~\ref{Theorem: Spectrum has no accumulation point}.
	Then $E_{\alpha,m} \subset L^p(\B)$ and the embedding 
	\[
	E_{\alpha,m} \embed L^p(\B)
	\]
	is compact.
\end{proposition}
Consequently, for $p \in (2,4)$ and $\alpha>1$ such that the 
conditions of Theorem~\ref{Theorem: Spectrum has no accumulation point} are 
satisfied, the map
\begin{equation}\label{eq:I_p-def}
I_p: E_{\alpha,m} \to \R, \qquad I_p(u) \coloneqq \frac{1}{p} \int_\B |u|^p \, 
dx
\end{equation}
is well-defined and continuous.

Next, we consider the decomposition
\[
E_{\alpha,m} = E_{\alpha,m}^+ \oplus E_{\alpha,m}^0 \oplus E_{\alpha,m}^-
\]
where $E_{\alpha,m}^+, E_{\alpha,m}^0,E_{\alpha,m}^-$ 
are given as the subspaces spanned by eigenfunctions 
corresponding to positive, zero, and negative eigenvalues of $L_{\alpha,m}$, 
respectively. 
In particular, for \(u, v \in C_c^\infty(\B)\) 
we may thus use integration by parts to write
\[
\langle L_{\alpha,m} u, v \rangle_{L^2(\B)} 
= \int_\B \left( \nabla u \cdot \nabla v - \alpha^2 (\del_\theta u)(\del_\theta v) 
+ m u v \right) \, dx
=\langle u^+, v \rangle_{\alpha,m} - \langle u^-, v \rangle_{\alpha,m}
\]
This motivates the definition of the energy functional
$\Phi_{\alpha,m}: E_{\alpha,m} \to \R$ as
\[
\begin{aligned}
	\Phi_{\alpha,m}(u) & \coloneqq 
    \frac{1}{2}\|u^+\|_{\alpha,m}^2 - \frac{1}{2} \|u^-\|_{\alpha,m}^2
	 - I_p(u) 
\end{aligned}
\]
with \(I_p\) given by \eqref{eq:I_p-def}. 
In particular, for \(u \in C_c^\infty(\B)\) integration by parts then 
lets us write
\[
\Phi_{\alpha,m}(u) = 
\frac{1}{2}  \int_\B \left( |\nabla u|^2 
- \alpha^2 |\del_\theta u|^2 + m u^2 \right) \, dx 
- \frac{1}{p} \int_\B |u|^p \, dx  \qquad \text{for $u \in C_c^\infty(\B)$.}
\]
Similarly, we may define weak solutions:

\begin{definition}
	Let $p \in (2,4)$, $m \in \R$ and assume $\alpha>1$ satisfies the 
	conditions of Theorem~\ref{Theorem: Spectrum has no accumulation point}.
	We define \textbf{weak solutions} of \eqref{Reduced equation} as functions $u \in 
	E_{\alpha,m}$ satisfying
	\[
	\int_\B |u|^{p-2} u \phi \, dx 
	= \langle u^+, \phi \rangle_{\alpha,m} - \langle u^-, \phi \rangle_{\alpha,m}
	%= \int_\B u L_{\alpha,m} \phi \, dx
	\]
	for all test functions $\phi \in C_c^\infty(\B)$.
\end{definition}
In particular, any critical 
point of $\Phi_{\alpha,m}$ is a weak solution, and we specifically define 
ground state solutions by considering the generalized Nehari manifold
\[
\cN_{\alpha,m} := \left\{
\begin{aligned}
u \in E_{\alpha,m} \setminus (E_{\alpha,m}^0 \oplus E_{\alpha,m}^-): \  
& \text{$\Phi_{\alpha,m}(u)'(u)u =0$ and $\Phi_{\alpha,m}(u)'(u)v = 0$ } \\
	& \text{for all $v \in E_{\alpha,m}^0 \oplus E_{\alpha,m}^-$}
\end{aligned}
\right\} .
\]
Note that $\cN_{\alpha,m} $ contains all nontrivial critical points of 
$\Phi_{\alpha,m}$. 
\begin{definition}
	Let $p \in (2,4)$, $m \in \R$ and assume $\alpha>1$ satisfies the 
	conditions of Theorem~\ref{Theorem: Spectrum has no accumulation point}.
	\begin{itemize}
		\item[(i)]
		We define the \textbf{ground state energy} as
		\[
		c_{\alpha,m} \coloneqq \inf_{u \in \cN_{\alpha,m}} \Phi_{\alpha,m }(u).
		\]
		In particular,
		$\Phi_{\alpha,m}(u) \geq c_{\alpha,m}$ for any critical point $u \in 
		E_{\alpha,m} \setminus \{0\}$.
		\item[(ii)]
		We say $u \in E_{\alpha,m}$ is a \textbf{ground state solution} of 
		\eqref{Reduced equation}, if $u$ is a critical point of 
		$\Phi_{\alpha,m}$ and satisfies $\Phi_{\alpha,m}(u) = c_{\alpha,m}$.
	\end{itemize}
\end{definition}
Importantly, the ground state energy admits a minimax representation using 
the cone
\[
\hat E _{\alpha,m}(u) \coloneqq \{t u +w: t \geq 0, w \in E^0 \oplus E^-\},
\]
as stated in the following central result.
\begin{proposition}
		Let $p \in (2,4)$, $m \in \R$ and assume $\alpha>1$ satisfies the 
	conditions of Theorem~\ref{Theorem: Spectrum has no accumulation point}.
	Then the ground state energy $c_{\alpha,m}$ is positive and attained by a 
	critical point of $\Phi_{\alpha,m}$, i.e., there exists a ground state 
	solution of~\eqref{Reduced equation}. Moreover, the minimax  
	characterization
	\[
	c_{\alpha,m} = \inf_{u \in E_{\alpha,m} \setminus (E^0 \oplus E^-)}
	\max_{w \in \hat E_{\alpha,m}(u)} \Phi_{\alpha,m}(w)
	\]
	holds.
\end{proposition}
%%%%%%%%%%%%%
%%%%%%%%%%
%%%%%%%%%%
In particular, this implies Theorem~\ref{theorem:intro:main_result2}(i).
We point out that these existence properties are essentially an application of 
well known results for indefinite problems due to Szulkin and Weth~\cite{Szulkin-Weth}.

The minimax characterization of $c_{\alpha,m}$ plays a vital role in proving 
the symmetry breaking result
stated in Theorem~\ref{theorem:intro:main_result2}(ii). More specifically, it 
allows for a comparison of $c_{\alpha,m}$ with the energy of the (up to sign 
unique, cf. \cite{McLeod-Serrin, Kwong, Kwong-Li}) radial solution
which is characterized as follows.
\begin{lemma}
	Let $\lambda_1>0$ denote the first Dirichlet eigenvalue of $-\Delta$ on $\B$
	and assume $p>2$, $m>-\lambda_1$. 
	Then there exists a unique radial solution $u_m \in H_{0,rad}^1(\B)$ of 
	\[
	\left\{ 
	\begin{aligned}
		-\Delta u + mu & = |u|^{p-2} u \qquad && \text{in $\B$,} \\
		u & = 0 && \text{on $\del \B$.}
	\end{aligned}
	\right.
	\]
	Furthermore, there exists $c>0$ such that
	\[
	\beta_m \coloneqq \Phi_{\alpha,m}(u_m) \geq c m^\frac{2}{p-2}
	\]
	holds for all $\alpha >1$ and $m \geq 0$.
\end{lemma}
On the other hand, it can be shown that the ground state energy satisfies
\[
c_{\alpha,m} \leq C_p m^\frac{p}{2(p-2)} 
\]
for some constant $C_p>0$ depending only on $p$. 
Noting that for $p \in (2,4)$ we have $\frac{p}{2(p-2)} < \frac{2}{p-2}$,
it follows that there exists $m_0 >0$ such that
\[
c_{\alpha,m} \leq C_p m^\frac{p}{2(p-2)} \leq c m^\frac{2}{p-2} < \beta_m
\]
holds for $m>m_0$. In particular, this yields  
Theorem~\ref{theorem:intro:main_result2}(ii).

\begin{remark}
	The existence and symmetry breaking result stated in 
	Theorem~\ref{theorem:intro:main_result2} is valid, in particular, whenever 
	$\alpha>1$ is given such that $\sigma =f^{-1}(\alpha) \in (0,\infty)$ is a 
	rational number with odd denominator. 
	Since such rational numbers are dense in $(0,\infty)$, Lemma~\ref{Lemma: 
	Monotonicity and Inverse for iota quotient} implies that the set of 
	associated values of $\alpha$ is in turn dense in $(1,\infty)$.
	
	For $\alpha \leq 1$, on the other hand, the existence of ground states was 
	established for $2 < p < 10$ in \cite{Kuebler-Weth}. For $2<p<4$ it is thus 
	natural to ask whether there is any connection between the ground states for
	$\alpha<1$ and $\alpha>1$, respectively. 
	More specifically, we would like to know if there exists sequences 
	$(\alpha_n)_n, (\alpha_n')_n$ such that $\alpha_n \uparrow 1, \alpha_n' 
	\downarrow 1$ such that the associated ground states converge to 
	the same solution for $\alpha=1$ in some sense. 
	
	We leave this question for future research.
\end{remark}
%
%
%
%
%%%%%%%%%%%%%%%%%%%%%%%%%%%%%%%%%%%%%%%%%%%%%%%%%%%%%%%%%%%%%%%%%%%%%%%%%%%%%%%%%%%%%%%%%%%%%%%%%%%%%%%%%%%%%%%%%
%%%%%%%%%%%%%%%%%%%%%%%%%%%%%%%%%%%%%%%%%%%%%%%%%%%%%%%%%%%%%%%%%%%%%%%%%%%%%%%%%%%%%%%%%%%%%%%%%%%%%%%%%%%%%%%%%
%%%%%%%%%%%%%%%%%%%%%%%%%%%%%%%%%%%%%%%%%%%%%%%%%%%%%%%%%%%%%%%%%%%%%%%%%%%%%%%%%%%%%%%%%%%%%%%%%%%%%%%%%%%%%%%%%
%

\end{document}